\documentclass[a4paper,12pt]{amsart}
\usepackage[utf8]{inputenc}
\usepackage[english]{babel}
\usepackage[a4paper,twoside,top=1in, bottom=1.1in, left=1in, right=1in]{geometry} 
\usepackage{graphicx}
\usepackage{hyperref}
\usepackage{amsmath}
\usepackage{amsthm}
\usepackage{amssymb}
\usepackage{esint} 
\usepackage{mathrsfs} 
\usepackage{xcolor}
\usepackage{array}
\usepackage{hhline}
\usepackage{enumitem} 
\usepackage{imakeidx} 
\usepackage{xparse} 
\usepackage{mathtools}

\usepackage{comment}

\definecolor{newblue}{rgb}{0.2, 0.3, 0.85}
\hypersetup{colorlinks=true, linkcolor=newblue, citecolor=newblue, urlcolor  = newblue} 

\usepackage{fancyhdr} 
\usepackage{ifthen} 
\usepackage{forloop} 
\usepackage{xstring} 
\usepackage{blindtext} 
\usepackage{nomencl}
\usepackage{dsfont} 
\usepackage{faktor} 
\usepackage{tikz}
\usetikzlibrary{arrows}

\numberwithin{equation}{section}

\usepackage[english,capitalize]{cleveref}

\usepackage[maxbibnames=99,backend=biber, sorting=nyt, doi=false, url=false]{biblatex}
\usepackage[colorinlistoftodos,textwidth=2.3cm]{todonotes}

\definecolor{dgreen}{rgb}{0.0, 0.56, 0.0}

\newcommand{\N}{\ensuremath{\mathbb N}}
\newcommand{\R}{\ensuremath{\mathbb R}}
\newcommand{\meas}{\mathfrak{m}}
\newcommand{\lip}{{\rm lip \,}}

\DeclarePairedDelimiter\abs{\lvert}{\rvert}

\DeclarePairedDelimiter\scal{\langle}{\rangle}

\newcommand{\st}{\ensuremath{\ :\ }} 
\newcommand{\eqdef}{\ensuremath{\vcentcolon=}}
\newcommand \eps{\ensuremath{\varepsilon}} 
\renewcommand{\epsilon}{\varepsilon}

\newcommand{\de}{\ensuremath{\,\mathrm d}} 
\renewcommand{\d}{\ensuremath{\mathrm d}} 
\DeclareMathOperator{\supp}{spt} 

\DeclareMathOperator{\sn}{sn}

\newcommand{\CD}{\mathsf{CD}}
\newcommand{\RCD}{\mathsf{RCD}}

\newcommand{\dist}{\mathsf{d}}

\let\div\undefined
\DeclareMathOperator{\div}{div}

\newcommand{\ric}{\ensuremath{\mathrm{Ric}}} 
\newcommand{\sect}{\ensuremath{\mathrm{Sect}}} 
\DeclareMathOperator{\vol}{vol}

\theoremstyle{plain}
\newtheorem{thm}{Theorem}[section] 
\theoremstyle{plain}

\theoremstyle{plain}
\newtheorem{prop}[thm]{Proposition}
\theoremstyle{plain}
\newtheorem{lemma}[thm]{Lemma}
\theoremstyle{plain}
\newtheorem{cor}[thm]{Corollary}
\theoremstyle{definition}
\newtheorem{defn}[thm]{Definition} 
\theoremstyle{definition}
\newtheorem{remark}[thm]{Remark}
\theoremstyle{definition}

\theoremstyle{definition}

\title[Existence of isoperimetric regions in $\ric\ge0$ and ${\rm AVR}>0$]{On the existence of isoperimetric regions in manifolds with nonnegative Ricci curvature and Euclidean volume growth}

\author{Gioacchino Antonelli}\address{Scuola Normale Superiore, Piazza dei Cavalieri, 7, 56126 Pisa, Italy.}\email{gioacchino.antonelli@sns.it}

\author{Elia Bru\`{e}}\address{Institute for Advanced Study, 1 Einstein Dr., Princeton NJ 05840,
U.S.A.}\email{elia.brue@math.ias.edu}

\author{Mattia Fogagnolo}\address{Centro di Ricerca Matematica Ennio De Giorgi, Scuola Normale Superiore, Piazza dei Cavalieri, 3, 56126 Pisa, Italy.}\email{mattia.fogagnolo@sns.it}

\author{Marco Pozzetta}\address{Dipartimento di Matematica e Applicazioni, Universit\`a di Napoli Federico II, Via Cintia, Monte S. Angelo, 80126 Napoli, Italy.}\email{marco.pozzetta@unina.it}

\date{\today}


\begin{document}

\maketitle

\begin{abstract}
In this paper we provide new existence results for isoperimetric sets of large volume in Riemannian manifolds with nonnegative Ricci curvature and Euclidean volume growth. We find sufficient conditions for their existence in terms of the geometry at infinity of the manifold. 

As a byproduct we show that isoperimetric sets of big volume always exist on manifolds with nonnegative sectional curvature and Euclidean volume growth.

Our method combines an asymptotic mass decomposition result for minimizing sequences, a sharp isoperimetric inequality on nonsmooth spaces, and the concavity property of the isoperimetric profile. The latter is new in the generality of noncollapsed manifolds with Ricci curvature bounded below.
\end{abstract}

\tableofcontents

\section{Introduction}

Given a Riemannian manifold $(M^n,g)$ of dimension $n\ge 2$ we denote by $\dist$, $\vol$, and $\ric$ the Riemannian distance, the volume measure, and the Ricci tensor of $(M^n,g)$, respectively.

For any $V \in [0,\vol(M^n))$, the \emph{isoperimetric problem} consists in studying the minimization problem
\begin{equation}
  I_{(M^n,g)}(V):=\inf\{P(\Omega):\text{$\Omega\subset  M^n$ set of finite perimeter with $\vol(\Omega)=V$}\} \, ,  
\end{equation}
where $P(\Omega)$ denotes the perimeter of $\Omega$, see \cref{sec:DefPreliminaryResults}. We shall drop the superscript $n$ in $(M^n,g)$ and the subscript $(M^n,g)$ in $I_{(M^n,g)}$ when there is no risk of confusion. The function $I_{(M^n,g)}$ is the so-called \textit{isoperimetric profile} of $(M^n,g)$. Any set of finite perimeter $E\subset M^n$ with $\vol(E)=V$ and $P(E)=I_{(M^n,g)}(E)$ is called \textit{isoperimetric set} or \textit{isoperimetric region}.

The problem of the existence of isoperimetric regions in the setting of noncompact manifolds is a hard problem that has seen several important progresses in the last years. Let us just mention some of the most important contributions in the field related to the topics of this work. Major results in the application of a direct method for proving existence of isoperimetric sets in manifolds with lower bounds on the Ricci curvature are contained \cite{Nar14, MondinoNardulli16}. The methods employed there have been generalized in \cite{AFP21}. As we are going to see, such generalization will be fundamental for the present work. In dimension $2$, a complete positive answer to the existence issue of isoperimetric sets under nonnegative curvature has been given in \cite{RitoreExistenceSurfaces01}. The existence of isoperimetric sets in $3$-manifolds with nonnegative scalar curvature and asymptotically flat asymptotics has been established in \cite{Carlotto2016}. Existence results for isoperimetric sets of large volumes in asymptotically flat manifolds were also obtained in \cite{EichmairMetzger, EichmairMetzger2}, in asymptotically hyperbolic spaces in \cite{Chodosh2016}, and in the asymptotically conical case in \cite{ChodoshEichmairVolkmann17}. When the ambient space is a nonnegatively Ricci curved cone, isoperimetric regions exist for any given volume and are actually characterized \cite{MorganRitore02}. 
In Euclidean solid cones, the problem has been investigated in \cite{RitRosales04}. In more general convex bodies, it is treated in \cite{LeonardiRitore}.

Differential properties of the isoperimetric profile, some of which are generalized in the present paper, have been studied in \cite{BavardPansu86, MorganJohnson00, Bayle03, Bayle04, BayleRosales}. A regularity theory for isoperimetric sets was established for example in \cite{Tama82, GonzalezMassariTamanini, Xia05, morgan2003regularity}, and, recently, it has been partly generalized in \cite{AntonelliPasqualettoPozzetta21}.

\subsection{Main existence results}

In this paper we provide new existence results for isoperimetric regions of large volume in the setting of manifolds with nonnegative Ricci curvature and \textit{Euclidean volume growth}, i.e., such that
\begin{equation*}
    \mathrm{AVR}(M^n,g):=\lim_{r\to +\infty}\frac{\vol(B_r(x))}{\omega_nr^n} \in (0,1] \, ,
\end{equation*}
where $\omega_n$ is the volume of the $n$-dimensional unit ball. The above limit exists as a consequence of the Bishop--Gromov monotonicity (cf.\! \cref{thm:BishopGromov}).
The same notion can be given for $\CD(0,n)$ metric measure spaces $(X,\dist,\meas)$ (replacing $\vol$ with $\meas$),
a class of nonsmooth spaces with a synthetic notion of nonnegative Ricci curvature and dimension bounded above by $n>0$, see \cref{sub:RCD}.

The Euclidean volume growth assumption implies that $(M^n,g)$ is \emph{noncollapsed}, i.e. there exists $v>0$ such that $\vol(B_1(x))\geq v$ for every $x\in M$. The latter can be shown to be equivalent to $I_{(M^n,g)}(V)\neq 0$ for some (and in fact for any) $V>0$ (cf.\! \cref{lem:NonSoDoveMetterlo}). Hence noncollapsedness is necessary for the existence of isoperimetric regions (see also the counterexamples discussed in \cite[Section 4.3]{AFP21}).

Before stating our first result we recall that $(X,\dist,\meas)$ is said to \emph{split (a line)} if it is isometric as a metric measure space to $(\R \times Y, \dist_\R \otimes \dist_Y, \meas_\R \otimes \meas_Y)$ for some $(Y,\dist_Y,\meas_Y)$. Here $\dist_{\R^k}, \meas_{\R^k}$ are just the Euclidean distance and the Lebesgue measure on $\R^k$, respectively, for any $k\in\N$.

\begin{thm}\label{thm:Principale} 
Let $(M^n,g)$ be a complete noncompact noncollapsed Riemannian manifold with $\ric\geq 0$. We write 
\begin{equation*}
  (M^n,g) =  (\mathbb R^k\times N^{n-k},g_{\mathbb R^k}+ g_N) \, ,
\end{equation*}
where $(N^{n-k}, g_N)$ does not split a line, and $0\le k\le n$.
Assume that there exists $\varepsilon>0$ such that every
pmGH limit $(X_\infty,\dist_\infty,\meas_\infty,x_\infty)$ of sequences $\{(M,\dist,\vol ,p_i)\}_{i\in\mathbb N_{\geq 0}}$ with $p_i:=(0,x_i)\in M$, and $\dist_N(x_i,x_0)\to+\infty$ as $i\to +\infty$, satisfies
$$
\mathrm{AVR}(X_\infty,\dist_\infty,\meas_\infty)\geq\mathrm{AVR}(M^n,g) + \varepsilon.
$$  
Then there exists $V_0>0$ such that for every $V\geq V_0$ there exists an isoperimetric region of volume $V$ in $M$.
\end{thm}

Recall that any sequence $\{(M,\dist,\vol ,p_i)\}_{i\in\mathbb N}$ admits limit points in the pmGH topology as a consequence of the Gromov compactness theorem (see \cref{def:GHconvergence} and \cref{rem:GromovPrecompactness}).

The hypothesis in \cref{thm:Principale} is verified when $\mathrm{AVR}(M^n,g)>0$ and the manifold $N$ possesses no asymptotic cones that split a line (see \cref{def:AsymptoticCone} for the notion of asymptotic cone). Hence we deduce the following.

\begin{thm}\label{cor:Princ2}
Let $(M^n,g)$ be a complete Riemannian manifold with $\ric\geq 0$ and $\mathrm{AVR}(M^n,g)>0$. We write 
\begin{equation*}
  (M^n,g) =  (\mathbb R^k\times N^{n-k},g_{\mathbb R^k}+ g_N) \, ,
\end{equation*}
where $(N^{n-k}, g_N)$ does not split a line, and $0\le k\le n$.  
If no asymptotic cone of $N$ splits a line, then there is $V_0>0$ such that for every $V\geq V_0$ there exists an isoperimetric region of volume $V$ in $M$.
\end{thm}

It is well-known that when $(M^n,g)$ has nonnegative sectional curvature
the asymptotic cone is unique and it splits if and only if the manifold splits. Since we were not able to find an explicit proof of this fact in the literature, we included one in \cref{lem:IfAndOnlyIfSplit}. This phenomenon marks an important difference with the nonnegative Ricci curvature, see  \cite[Theorem 1.4]{ColdingNaber} and \cref{subsec:example} below. In turn, we derive the following existence theorem.

\begin{cor}\label{thm1}
Let $(M^n,g)$ be a complete Riemannian manifold with nonnegative sectional curvature and Euclidean volume growth. Then there exists $V_0>0$ such that for every $V\geq V_0$ there exists an isoperimetric region of volume $V$.
\end{cor}

\subsection{Concavity of the isoperimetric profile}

We now state our last result regarding concavity properties of the isoperimetric profile. 
This will be an important ingredient in the proof of the previous existence theorems and has an independent interest.

\begin{thm}\label{thm:Concavity}
Suppose that $(M^n,g)$ is noncollapsed and that $\ric \ge (n-1)K$ for $K \le 0$.
Then for any $\widetilde V>0$ there is $\delta>0$ and $C\ge0$ such that the function $V\mapsto I_{(M^n,g)}(V) - CV^2$ is concave on $(\widetilde V- \delta,\widetilde V+ \delta)$. Moreover, $I_{(M^n,g)}$ is twice differentiable almost everywhere and at any point $V$ of twice differentiability it holds the estimate
\begin{equation}\label{eq:EstimateI''}
    I_{(M^n,g)}''(V) \le \frac{(n-1)|K|}{I_{(M^n,g)}(V)}.
\end{equation}
If $K=0$ then $I_{(M^n,g)}$ is concave on $(0,+\infty)$.
\end{thm}

The previous result is obtained under the natural assumptions for the study of the isoperimetric problem, namely noncollapsedness and Ricci bounded below. In particular, differently from the classical results in \cite{BavardPansu86, MorganJohnson00, Bayle03, Bayle04, BayleRosales, MondinoNardulli16}, the existence of isoperimetric regions is not assumed \emph{a priori}, and, in fact, we will employ \cref{thm:Concavity} as a tool to provide existence. In order to prove \cref{thm:Concavity}, we introduce an approximation of the perimeter functional on the manifold $M$ by a sequence of penalized perimeters $P_k$ given by the sum of the usual perimeter and a potential term. The penalization in the definition of $P_k$, see \eqref{eq:DefPk}, implies the existence of volume constrained minimizers for $P_k$ for any $k$. Also, the profile $I_k$ corresponding to $P_k$, see \eqref{eq:DefIk}, converges to the isoperimetric profile $I_{(M^n,g)}$ of the manifold locally uniformly. Hence \cref{thm:Concavity} eventually follows by studying the concavity properties of $I_k$ by means of the existence of minimizers for $P_k$, and then by passing the estimates to the limit with respect to $k$.

As a consequence of \cref{thm:Concavity}, we get that under the same assumptions the isoperimetric profile is a locally Lipschitz function (\cref{cor:loclip}), which improves the result of local H\"{o}lder continuity contained in \cite{FloresNardulli20}.

As anticipated, \cref{thm:Concavity} has useful consequences for the the existence of isoperimetric regions in manifolds with nonnegative Ricci curvature. This is due to the fact that the concavity of the profile provides information on the asymptotic behavior of $I_{(M^n,g)}$ and $I'_{(M^n,g)}$, see \cref{lem:IsopAVR}. Under the additional assumption of Euclidean volume growth, we can further conclude that $I_{(M^n,g)}$ is strictly increasing (\cref{cor:ProfileIncreasing}). The latter result has to be compared with \cite{Rit17}, where the author considers hypotheses on the sectional curvature instead, but without assumptions on the asymptotic volume ratio.

\subsection{Strategy of proof of the existence results}
Broadly speaking the proof of \cref{thm:Principale} is based on the direct method of the Calculus of Variations and a compensated compactness argument.
Given $V>0$ we consider a minimizing sequence $E_i$ such that $\vol(E_i)=V$ and $\lim_{i\to \infty}P(E_i) = I(V)$, and we aim at studying limits in a generalized sense. Notice that the existence of a limit in $L^1(M^n,\vol)$ would immediately give the existence of isoperimetric regions of volume $V$, as a consequence of the lower semicontinuity of the perimeter.
The main enemy for the existence of the limit is the possibility of having pieces of $E_i$ escaping at infinity. 

The asymptotic decomposition result in \cite{AFP21} (see \cref{thm:MassDecomposition} below)
tells us that, for noncollapsed Riemannian manifolds with Ricci curvature bounded from below, the mass lost at infinity can be recovered by looking at subsets of metric measure spaces obtained as pmGH limit points of $(M^n,\dist, \vol, p_i)$, where $p_i\to \infty$ is a suitable sequence. 
The hypotheses in \cref{thm:Principale} are meant to control the isoperimetric profile of these limit spaces, with the aim of showing that escaping at infinity is not ``isoperimetrically convenient''. 
To this aim we employ the recent result \cite[Theorem 1.1]{BaloghKristaly}, showing that the isoperimetric constant on $\CD(0, n)$ spaces is, roughly speaking, realized by balls of infinite radius, and thus that it is explicitly related to the asymptotic volume ratio. This result generalizes earlier analogous inequalities in the smooth setting \cite{AgostinianiFogagnoloMazzieri,BrendleFigo,FogagnoloMazzieri,Johne}. In \cite{BaloghKristaly}, the isoperimetric inequality is actually obtained as an a priori weaker estimate on the Minkowksi content of the sets involved; in this paper, we give a refinement of that in terms of the perimeter, see \cref{thm:IsoperimetriSharpBrendle}. Notice that the pmGH limits at infinity of a noncollpased nonnegatively Ricci curved manifold are indeed $\CD(0,n)$ spaces. Our assumptions in  \cref{thm:Principale}, in particular, imply that for large enough volumes it is more convenient for the minimizing sequence not to lose mass at infinity along $N$. On the other hand, if part of the mass of any minimizing sequence is lost at infinity along the factor $\mathbb R^k$, it can be taken back by means of translations along the Euclidean factor.

To perform such a plan we need to know the exact asymptotic behaviour for large volumes of the isoperimetric profile $I$ and its derivative $I'$, for an arbitrary manifold with nonnegative Ricci curvature. The latter piece of information is used to ensure that every minimizing sequence of sets of sufficiently big volumes can lose at most one piece at infinity along $N$, see \cref{lem:Step1}. 

Loosely speaking, the behaviour at infinity of $I$ is the same as the isoperimetric profile of a cone with opening $\mathrm{AVR}(M^n,g)$. The computation of the asymptotics of $I$ is a direct consequence of Bishop--Gromov comparison theorem and the sharp isoperimetric inequality, while the computation of the asymptotics of $I'$ is a byproduct of the concavity property proven in \cref{thm:Concavity}. 

We finally stress that the extremal case in the hypothesis in \cref{thm:Principale}, i.e., when $\mathrm{AVR}(X_{\infty})=1$, happens when every limit $X_\infty$ is isometric to $\mathbb R^n$. In this simpler case, since the mass of a minimizing sequence lost along the factor $\mathbb R^k$ can be taken back as mentioned above, by a minor variation of the proof of \cite[Theorem 5.2]{AFP21} one even gets existence of isoperimetric regions for every volume.

\bigskip

Let us comment on the proof of \cref{cor:Princ2}. We show that, under the hypotheses of \cref{cor:Princ2}, the hypotheses of \cref{thm:Principale} are satisfied. This is done in two steps, which correspond to item (i) and (ii) of \cref{lem:DensitaStaccata}. First, if no asymptotic cone of $N$ splits a line, then the density of every point at distance $1$ from the line of the tips of every asymptotic cone of $M$ is uniformly bigger than $\mathrm{AVR}(M^n,g)$. The latter is due to a compactness argument and to the fact that if a point at distance $1$ from the line of the tips of an asymptotic cone $\mathbb R^k\times C$ of $M$ has density equal to $\mathrm{AVR}(M^n,g)$, then there is a line in $\mathbb R^k\times C$ which is not contained in $\mathbb R^k$ and then, by the splitting theorem, one has a splitting of the cone $C$, which is an asymptotic cone of $N$, resulting in a contradiction. Second, by means of the volume convergence theorem, a lower bound on the density of points at distance $1$ from the line of the tips of the asymptotic cones is readily seen to imply a lower bound on the $\mathrm{AVR}$ of the pmGH limits at infinity of the manifold along $N$.

The result in item (i) of \cref{lem:DensitaStaccata} discussed above allows to get also other nontrivial existence results for the isoperimetric problem. It can be proved, see \cref{thm:Further1}, that for manifolds as in \cref{cor:Princ2}, if every point at distance $1$ from the line of the tips of every asymptotic cone is regular, i.e., has density one, then isoperimetric regions exist for every volume. In particular this implies, see \cref{cor:FurtherExistence}, that if a Riemannian manifold with nonnegative Ricci curvature and Euclidean volume growth is such that every asymptotic cone has a smooth cross section, then isoperimetric regions exist for every volume. We point out that when the dimension is $2$ this partially recovers the aforementioned existence result of isoperimetric sets in nonnegative curvature due to Ritoré \cite{RitoreExistenceSurfaces01}.

\subsection{Sharpness and counterexamples}\label{subsec:example}

While the assumptions in \cref{thm:Principale} turn out to be always satisfied in the setting of manifolds with nonnegative sectional curvature, giving rise to \cref{thm1}, the following example borrowed from \cite[pp. 913-914]{KasueWashio} shows that this is not always the case on manifolds with nonnegative Ricci curvature. 
Let us consider the metric
\begin{equation*}
 g:=f(r)^2 dt^2 + dr^2 + \eta(r)^2 g_{\mathbb{S}^{k-1}}   
 \quad
 \text{for $(t,r,p) \in \R\times [0,+\infty) \times \mathbb{S}^{k-1} \simeq \R^{k+1}$, $k\ge 3$} \, ,
\end{equation*}
where
\begin{equation*}
   f(r) = (b + r^2)^{(\beta+2 - k)/2} + c \, 
   \quad 
   \alpha, \beta \in (0,1), \, b, c>1\, ,\,  k-\beta > 2 + \alpha\, ,
\end{equation*}
\begin{equation*}
    \eta(r) = \frac{1}{2}r + \frac{1}{2a} \int_0^r \int_t^\infty \xi(s) d s \, ,
\end{equation*}
for some smooth $\xi:[0, \infty) \to [0, \infty)$ satisfying $\xi(t) = t$ on $[0,1]$, $\xi >0$ on $[1,2]$ and $\xi(t)=t^{-1-\alpha}$ on $[2, \infty)$. The parameter $a>0$ is chosen such that $\eta'(0) = 1$.

It is possible to prove that $g$ has nonnegative Ricci curvature provided $b$ and $c$ are chosen big enough. It is easy to see that $(\R^{k+1}, g)$ has Euclidean volume growth and admits a unique asymptotic cone isometric to $\R \times C(\mathbb{S}^{k-1}_\rho)$, where $\rho= 1/2$ is the radius of $\mathbb{S}^{k-1}_\rho$, which is the cross section of the cone $C(\mathbb{S}^{k-1}_\rho)$. Observe that the asymptotic cone splits a line while $(\R^{k+1}, g)$ does not. Moreover, $g$ is invariant under the translation along the direction $\partial_t$. In particular, if $P_i=(t_i,r_i,p_i)$ satisfies $\sup_i r_i <+\infty$ and $|t_i| \to \infty$, we deduce that $(\R^{k+1}, \dist_g, \vol_g, P_i) \to (\R^{k+1}, \dist_g, \vol_g, 0)$ in the pmGH topology. Hence the assumptions of \cref{thm:Principale} are \emph{not} satisfied.

On the other hand, along sequences $Q_i=(t_i,r_i,p_i)$ with $r_i\to\infty$, one checks that $(\R^{k+1}, \dist_g, \vol_g,  Q_i)$ converges to the Euclidean space $(\R^{k+1}, \dist_{\rm eu}, \mathcal{L}^{k+1}, 0)$ in the pmGH topology. Therefore, by the asymptotic mass decomposition \cref{thm:MassDecomposition}, one has that the volume of a minimizing sequence lost at infinity ends up in isoperimetric regions in $(\R^{k+1},g_{\rm eu})$ or in $(\R^{k+1}, g)$. However, the presence of such limit regions in $(\R^{k+1},g_{\rm eu})$ easily contradicts the minimizing property of the initial sequence (cf. \cite[Theorem 5.1 \& 5.2]{AFP21} and \cite[Theorem 3.5]{MorganJohnson00}), while limit isoperimetric regions in $(\R^{k+1}, g)$ can be obviously brought back into the original $(\R^{k+1}, g)$, recovering the missing volume. Therefore we obtain that $(\R^{k+1}, g)$ has isoperimetric regions for \emph{any} volume.

All in all, we conclude that our assumptions in \cref{thm:Principale} are not necessary for the existence of isoperimetric sets in the setting of manifolds with nonnegative Ricci curvature.

To date we do not know whether manifolds with nonnegative Ricci curvature and Euclidean volume growth always admit isoperimetric regions for big volumes, and we plan to investigate further this problem in the future.

\subsection{Plan of the paper} 
In \cref{sec:DefPreliminaryResults} we introduce the preliminary material. We briefly describe the class of $\CD$ and $\RCD$ spaces, we recall the notion of pmGH convergence and the mass decomposition theorem (\cref{thm:MassDecomposition}), and we prove the mentioned sharp Sobolev inequality (\cref{thm:IsoperimetriSharpBrendle}). \cref{sec:Concavity} is devoted to the proof of \cref{thm:Concavity}. In \cref{sec:SectionExistence} we prove the existence results \cref{thm:Principale}, \cref{cor:Princ2}, \cref{thm1}, and further existence results (\cref{sec:Further}).

\subsection*{Acknowledgments} 
The first author is partially supported by the European Research Council (ERC Starting Grant 713998 GeoMeG `\emph{Geometry of Metric Groups}'). The second author is supported by the Giorgio and Elena Petronio Fellowship at the Institute for Advanced Study. He is grateful to Stefano Nardulli for having introduced him to the isoperimetric problem on noncompact manifolds. The authors also thank Gian Paolo Leonardi, Aaron Naber, and Manuel Ritor\'{e} for fruitful discussions and comments.

\section{Definitions and auxiliary results}\label{sec:DefPreliminaryResults}

\subsection{Preliminaries on Riemannian manifolds and metric measure spaces}
In this preliminary section we introduce basic facts and notations about Riemannian manifolds and metric measure spaces. 

Given two Riemannian manifolds $(M^m,g), (N^n,h)$, we denote with $(M^m\times N^n,g+h)$ the product of the two Riemannian structure. 
For the notions of BV and Sobolev spaces on Riemannian manifolds we refer the reader to \cite[Section 1]{MirandaPallaraParonettoPreunkert07}.
For every finite perimeter set $E$ in $\Omega$ we denote with $P(E,\Omega)$ the perimeter of $E$ inside $\Omega$. When $\Omega=M^n$ we simply write $P(E)$. We denote with $\mathcal{H}^{n-1}$ the $(n-1)$-dimensional Hausdorff measure on $M^n$ relative to the distance induced by $g$. We recall that for every finite perimeter set $E$ one has $P(E)=\mathcal{H}^{n-1}(\partial^*E)$ and the characteristic function $\chi_E$ belongs to $BV_{\rm loc}(M^n, \vol)$ with generalized gradient $D\chi_E = \nu \mathcal{H}^{n-1}\llcorner \partial^* E$ for a function $\nu:M\to T M^n$ with $|\nu|=1$ at $|D\chi_E|$-a.e. point, where $\partial^* E$ is the essential boundary of $E$.

We recall the following terminology.

\begin{defn}[Convergence of finite perimeter sets]\label{def:ConvergenceFinitePerimeter}
Let $(M^n,g)$ be a Riemannian manifold. We say that a sequence of measurable (with respect to the volume measure) sets $E_i$ \emph{locally converges} to a measurable set $E$ if the characteristic functions $\chi_{E_i}$ converge to $\chi_E$ in $L^1_{\rm loc}(M^n,g)$. In such a case we simply write that $E_i\to E$ locally on $M^n$.

If the sets $E_i$ have also locally finite perimeter, that is, $P(E_i,\Omega)<+\infty$ for any $i$ and any bounded open set $\Omega$, we say that $E_i\to E$ \emph{in the sense of finite perimeter sets} if $E_i\to E$ locally on $M^n$ and the sequence of measures $D\chi_{E_i}$ locally weakly* converges as measures, that is, with respect to the duality with compactly supported continuous functions. In such a case, $E$ has locally finite perimeter and the weak* limit of $D\chi_{E_i}$ is $D\chi_E$.
\end{defn}

\begin{remark}[Approximation of finite perimeter sets with smooth sets]\label{rem:Approximation}
It can be proved, see \cite[Lemma 2.3]{FloresNardulli20}, that when $M^n$ is a complete Riemannian manifold every finite perimeter set $\Omega$ with $0<\vol(\Omega)<+\infty$ and $\vol(\Omega^c)>0$ is approximated by relatively compact sets $\Omega_i$ in $M^n$ with smooth boundary such that $\vol(\Omega_i)=\vol(\Omega)$ for every $i\in\mathbb N$, $\vol(\Omega_i\Delta\Omega)\to 0$ when $i\to +\infty$, and $P(\Omega_i)\to P(\Omega)$ when $i\to +\infty$. Thus, by approximation, one can deduce that 
$$
I(V)=\inf\{\mathcal{H}^{n-1}(\partial\Omega):\text{$\Omega \Subset M^n$ has smooth boundary, $\vol(\Omega)=V$}\},
$$
see \cite[Theorem 1.1]{FloresNardulli20}.
\end{remark}

We also need to recall the definition of the simply connected radial models with constant sectional curvature.

\begin{defn}[Models of constant sectional curvature, cf. {\cite[Example 1.4.6]{Petersen2016}}]\label{def:Models}
Let us define
\[
\sn_K(r) := \begin{cases}
(-K)^{-\frac12} \sinh((-K)^{\frac12} r) & K<0,\\
r & K=0,\\
K^{-\frac12} \sin(K^{\frac12} r) & K>0.
\end{cases}
\]

If $K>0$, then $((0,\pi/\sqrt{K}]\times \mathbb S^{n-1},\d r^2+\mathrm{sn}_K^2(r)g_1)$, where $g_1$ is the canonical metric on $\mathbb S^{n-1}$, is the radial model of dimension $n$ and constant sectional curvature $K$. The metric can be smoothly extended at $r=0$, and thus we shall write that the the metric is defined on the ball $\mathbb B^n_{\pi/\sqrt{K}} \subset \R^n$. The Riemannian manifold $(\mathbb B^n_{\pi/\sqrt{K}}, g_K\eqdef \d r^2+\mathrm{sn}_K^2(r)g_1)$ is the unique (up to isometry) simply connected Riemannian manifold of dimension $n$ and constant sectional curvature $K>0$.

If instead $K\leq 0$, then $((0,+\infty)\times\mathbb S^{n-1},\d r^2+\mathrm{sn}_K^2(r)g_1)$ is the radial model of dimension $n$ and constant sectional curvature $K$. Extending the metric at $r=0$ analogously yields the unique (up to isometry) simply connected Riemannian manifold of dimension $n$ and constant sectional curvature $K\leq 0$, in this case denoted by $(\R^n,g_K)$.

We denote by $v(n,K,r)$ the volume of the ball of radius $r$ in the (unique) simply connected Riemannian manifold of sectional curvature $K$ of dimension $n$, and by $s(n, K, r)$ the volume of the boundary of such a ball. In particular $s(n,K,r)=n\omega_n\mathrm{sn}_K^{n-1}(r)$ and $v(n,K,r)=\int_0^rn\omega_n\mathrm{sn}_K^{n-1}(r)\de r$, where $\omega_n$ is the Euclidean volume of the Euclidean unit ball in $\mathbb R^n$.

On $\mathbb R^n$ we denote with $\dist_{\mathbb R^n}$, $g_{\mathbb R^n}$, and $\meas_{\mathbb R^n}$, respectively, the Euclidean distance, the Euclidean metric, and the Lebesgue measure.
\end{defn}

Let us now briefly recall the main concepts we will need from the theory of metric measure spaces. We recall that a {\em metric measure space, $\mathrm{m.m.s.}$ for short,} $(X,\dist_X,\mathfrak{m}_X)$ is a triple where $(X,\dist_X)$ is a locally compact separable metric space and $\mathfrak{m}_X$ is a Borel measure bounded on bounded sets. A {\em pointed metric measure space} is a quadruple $(X,\dist_X,\mathfrak{m}_X,x)$ where $(X,\dist_X,\mathfrak{m}_X)$ is a metric measure space and $x\in X$ is a point. 
\begin{center}
    For simplicity, and since it will always be our case, we will always assume that given $(X,\dist_X,\meas_X)$ a m.m.s.\! the support ${\rm spt}\,\meas_X$ of the measure $\meas_X$ is the whole $X$.
\end{center}

Given two m.m.s. $(X,\dist_X,\meas_X)$ and $(Y,\dist_Y,\meas_Y)$, we denote by $(X\times Y,\dist_X\otimes\dist_Y,\meas_X\otimes\meas_Y)$ the product m.m.s., where 
$$
\dist_X\otimes\dist_Y((x,y),(x',y')):=\sqrt{\dist_X(x,x')^2+\dist_Y(y,y')^2}, \qquad \forall x,x'\in X,\quad\forall y,y'\in Y,
$$
and $\meas_X\otimes\meas_Y$ is the usual product of measures. 

We assume the reader to be familiar with the notion of pointed measured Gromov--Hausdorff convergence, referring to \cite[Chapter 27]{VillaniBook} and to \cite[Chapter 7 and 8]{BuragoBuragoIvanovBook} for an overview on the subject. In the following treatment we introduce the pmGH-convergence already in a proper realization even if this is not the general definition. Nevertheless, the (simplified) definition of Gromov--Hausdorff convergence via a realization is equivalent to the standard definition of pmGH convergence in our setting, because in the applications we will always deal with locally uniformly doubling measures, see \cite[Theorem 3.15 and Section 3.5]{GigliMondinoSavare15}. The following definition is taken from the introductory exposition of \cite{AmborsioBrueSemola19}.

\begin{defn}[pGH and pmGH convergence]\label{def:GHconvergence}
A sequence $\{ (X_i, \dist_i, x_i) \}_{i\in \N}$ of pointed metric spaces is said to converge in the \emph{pointed Gromov--Hausdorff topology, in the $\mathrm{pGH}$ sense for short,} to a pointed metric space $ (Y, \dist_Y, y)$ if there exist a complete separable metric space $(Z, \dist_Z)$ and isometric embeddings
\[
\begin{split}
&\Psi_i:(X_i, \dist_i) \to (Z,\dist_Z), \qquad \forall\, i\in \N,\\
&\Psi:(Y, \dist_Y) \to (Z,\dist_Z),
\end{split}
\]
such that for any $\eps,R>0$ there is $i_0(\varepsilon,R)\in\mathbb N$ such that
\[
\Psi_i(B_R^{X_i}(x_i)) \subset \left[ \Psi(B_R^Y(y))\right]_\eps,
\qquad
\Psi(B_R^{Y}(y)) \subset \left[ \Psi_i(B_R^{X_i}(x_i))\right]_\eps,
\]
for any $i\ge i_0$, where $[A]_\eps\eqdef \{ z\in Z \st \dist_Z(z,A)\leq \eps\}$ for any $A \subset Z$.

Let $\meas_i$ and $\mu$ be given in such a way $(X_i,\dist_i,\meas_i,x_i)$ and $(Y,\dist_Y,\mu,y)$ are m.m.s.\! If in addition to the previous requirements we also have $(\Psi_i)_\sharp\mathfrak{m}_i \rightharpoonup \Psi_\sharp \mu$ with respect to duality with continuous bounded functions on $Z$ with bounded support, then the convergence is said to hold in the \emph{pointed measure Gromov--Hausdorff topology, or in the $\mathrm{pmGH}$ sense for short}.
\end{defn}

We remark that in the setting we will deal with, product structures are stable under pmGH convergence, i.e., if $(X_n,\dist_{X_n},\meas_{X_n},x_n)\to (X,\dist_X,\meas_X,x)$ and $(Y_n,\dist_{Y_n},\meas_{Y_n},y_n)\to (Y,\dist_Y,\meas_Y,y)$ in pmGH then 
$$
(X_n\times Y_n,\dist_{X_n}\otimes\dist_{Y_n},\meas_{X_n}\otimes\meas_{Y_n}, (x_n,y_n))\to (X\times Y,\dist_X\otimes\dist_Y,\meas_X\otimes\meas_Y,(x,y)),
$$
in pmGH.

\subsection{$\RCD$ spaces}\label{sub:RCD}
Let us briefly introduce the so-called $\RCD$ condition for m.m.s., and discuss some basic and useful properties of it. Since we will use part of the $\RCD$ theory just as an instrument for our purposes and since we will never use in the paper the specific definition of $\RCD$ space, we just outline the main references on the subject and we refer the interested reader to the survey of Ambrosio \cite{AmbrosioSurvey} and the references therein. 

After the introduction, in the independent works \cite{Sturm1,Sturm2} and \cite{LottVillani}, of the curvature dimension condition $\CD(K,n)$ encoding in a synthetic way the notion of Ricci curvature bounded from below by $K$ and dimension bounded above by $n$, the definition of $\RCD(K,n)$ m.m.s.\! was first proposed in \cite{GigliRCD} and then studied in \cite{Gigli13, ErbarKuwadaSturm15,AmbrosioMondinoSavare15}, see also \cite{CavallettiMilman16} for the equivalence between the $\RCD^*(K,n)$ and the $\RCD(K,n)$ condition. The infinite dimensional counterpart of this notion had been previously investigated in \cite{AmbrosioGigliSavare14}, see also \cite{AmbrosioGigliMondinoRajala15} for the case of $\sigma$-finite reference measures. 

\begin{remark}[pmGH limit of $\RCD$ spaces]\label{remark:stability}
	We recall that, whenever it exists, a pmGH limit of a sequence $\{(X_i,\dist_i,\meas_i,x_i)\}_{i\in\mathbb N}$ of (pointed) $\RCD(K,n)$ spaces is still an $\RCD(K,n)$ metric measure space.
\end{remark}

Due to the compatibility of the $\RCD$ condition with the smooth case of Riemannian manifolds with Ricci curvature bounded from below and to its stability with respect to pointed measured Gromov--Hausdorff convergence, limits of smooth Riemannian manifolds with Ricci curvature uniformly bounded from below by $K$ and dimension uniformly bounded from above by $n$ are $\RCD(K,n)$ spaces. Then the class of $\RCD$ spaces includes the class of Ricci limit spaces, i.e., limits of sequences of Riemannian manifolds with the same dimension and with Ricci curvature uniformly bounded from below \cite{ChCo0,ChCo1,ChCo2,ChCo3}.
An extension of noncollapsed Ricci limit spaces is the class of $\RCD(K,n)$ space where the reference measure is the $n$-dimensional Hausdorff measure relative to the distance, introduced and studied in \cite{Kitabeppu17, DePhilippisGigli18, AntBruSem19}. As a consequence of the rectifiability of $\RCD$ spaces \cite{MondinoNaber, BruePasqualettoSemolaRectifiabilitySpace2020}, we remark that if $(X,\dist,\mathcal{H}^n)$ is an $\RCD(K,n)$ space then $n$ is an integer.

We state the volume convergence theorems obtained by Gigli and De Philippis in \cite[Theorem 1.2 and Theorem 1.3]{DePhilippisGigli18}, which are the synthetic version of the celebrated volume convergence of Colding \cite{Colding97}.

\begin{thm}\label{thm:volumeconvergence}
	Let $\{(X_i,\dist_i,\mathcal{H}^n,x_i)\}_{i\in\mathbb N}$ be a sequence of pointed $\RCD(K,n)$ m.m.s.\! with $K\in\mathbb{R}$ and $n\in [1,+\infty)$. Assume that $(X_i,\dist_i,x_i)$ converges in the pGH topology to $(X,\dist,x)$. Then precisely one of the following happens
	\begin{itemize}
		\item[(a)] $\limsup_{i\to\infty}\mathcal{H}^n\left(B_1(x_i)\right)>0$. Then the $\limsup$ is a limit and $(X_i,\dist_i,\mathcal{H}^n,x_i)$ converges in the pmGH topology to $(X,\dist,\mathcal{H}^n,x)$. Hence $(X,\dist,\mathcal{H}^n)$ is an $\RCD(K,n)$ m.m.s. endowed with the $n$-dimensional Hausdorff measure;
		\item[(b)] $\lim_{i\to\infty}\mathcal{H}^n(B_1(x_i))=0$. In this case we have $\dim_{H}(X,\dist)\le n-1$, where we denoted by $\dim_H(X,\dist)$ the Hausdorff dimension of $(X,\dist)$. 
	\end{itemize}
	Moreover, for $K\in\mathbb R$ and $n\in[1,+\infty)$, let $\mathbb B_{K,n,R}$ be the collection of all equivalence classes up to isometry of closed balls of radius $R$ in $\RCD(K,n)$ spaces, equipped with the Gromov-Hausdorff distance. Then the map $\mathbb B_{K,n,R}\ni Z\to \mathcal{H}^n(Z)$ is real-valued and continuous.
\end{thm}

\begin{remark}[Gromov precompatness theorem for $\RCD$ spaces]\label{rem:GromovPrecompactness}
Here we recall the synthetic variant of Gromov's precompactness theorem for $\RCD$ spaces, see \cite[Equation (2.1)]{DePhilippisGigli18}. Let $\{(X_i,\dist_i,\meas_i,x_i)\}_{i\in\mathbb N}$ be a sequence of $\RCD(K_i,n)$ spaces with $n\in[1,+\infty)$, $\supp(\meas_i)=X_i$ for every $i\in\mathbb N$, $\meas_i(B_1(x_i))\in[v,v^{-1}]$ for some $v\in(0,1)$ and for every $i\in\mathbb N$, and $K_i\to K\in\mathbb R$. Then there exists a subsequence pmGH-converging to some $\RCD(K,n)$ space $(X,\dist,\meas,x)$ with $\supp(\meas)=X$.
\end{remark}

We conclude this part by recalling a few basic definitions and results concerning the perimeter functional in the setting of metric measure spaces (see \cite{Ambrosio02,Miranda03,AmbrosioDiMarino14}).

\begin{defn}[$BV$ functions and perimeter on m.m.s.]\label{def:BVperimetro}
Let $(X,\dist,\meas)$ be a metric measure space. A function $f \in L^1(X,\meas)$ is said to belong to the space of \emph{bounded variation functions} $BV(X,\dist,\meas)$ if there is a sequence $f_i \in {\rm Lip}_{\mathrm{loc}}(X)$ such that $f_i \to f$ in $L^1(X,\meas)$ and $\limsup_i \int_X \lip f_i \de \meas < +\infty$, where $\lip u (x) \eqdef \limsup_{y\to x} \frac{|u(y)-u(x)|}{\dist(x,y)}$ is the \emph{slope} of $u$ at $x$, for any accumulation point $x\in X$, and $\lip u(x):=0$ if $x\in X$ is isolated. In such a case we define
\[
|Df|(A) \eqdef \inf\left\{\liminf_i \int_A \lip f_i \de\meas \st \text{$f_i \in {\rm Lip}_{\rm loc}(A), f_i \to f $ in $L^1(A,\meas)$} \right\},
\]
for any open set $A\subset X$.

If $E\subset X$ is a Borel set and $A\subset X$ is open, we  define the \emph{perimeter $P(E,A)$  of $E$ in $A$} by
\[
P(E,A) \eqdef \inf\left\{\liminf_i \int_A \lip u_i \de\meas \st \text{$u_i \in {\rm Lip}_{\rm loc}(A), u_i \to \chi_E $ in $L^1_{\rm loc}(A,\meas)$} \right\},
\]
We say that $E$ has \emph{finite perimeter} if $P(E,X)<+\infty$, and we denote by $P(E)\eqdef P(E,X)$. Let us remark that the set functions $|Df|, P(E,\cdot)$ above are restrictions to open sets of Borel measures that we denote by $|Df|, |D\chi_E|$ respectively, see \cite{AmbrosioDiMarino14}, and \cite{Miranda03}. 

The \emph{isoperimetric profile of $(X,\dist,\meas)$} is
\[
I_{X}(V)\eqdef \inf \left\{ P(E) \st \text{$E\subset X$ Borel, $\meas(E)=V$} \right\},
\]
for any $V\in[0,\meas(X))$. If $E\subset X$ is Borel with $\meas(E)=V$ and $P(E)=I_X(V)$, then we say that $E$ is an \emph{isoperimetric region}.
\end{defn}

It follows from classical approximation results (cf. \cref{rem:Approximation}) that the above definition yields the usual notion of perimeter on any Riemannian manifold $(M^n,g)$ recalled at the beginning of this section.

We will need the following well established approximation result.

\begin{lemma}\label{lemma:approximationSobolevfunctio}
Let $(X,\dist)$ be a complete and separable metric space, and let $\meas$ be a nonnegative measure, finite on bounded sets. Then, for any $f\in BV(X,\dist,\meas)$ there exists a sequence $(f_k)\subset  {\rm Lip}(X,\dist)$, where $f_k$ has bounded support for any $k$, such that $f_k \to f$ pointwise $\meas$-a.e. and in $L^1(X,\meas)$, and $\int_X \lip(f_k) d \meas \to |Df|(X)$ as $k\to \infty$. 
\end{lemma}

\begin{proof}
For $f \in BV(X,\dist,\meas)$ it holds that
\[
\begin{split}
    |Df|(X) & = |Df|_w(X) \\
    &= \inf \left\{ \liminf_i \int_X {\rm lip}_a\, f_i \de \meas \st f_i \in {\rm Lip}(X) \text{ with bdd spt, } f_i\to f \text{ in } L^1(X,\meas) \right\} \\
    &= \inf \left\{ \liminf_i \int_X {\rm lip}_a\, f_i \de \meas \st f_i \in {\rm Lip}_{\rm loc}(X), f_i\to f \text{ in } L^1(X,\meas) \right\}
    \ge |Df|(X),
\end{split}
\]
where $|Df|_w$ is as in \cite[Section 5.3]{AmbrosioDiMarino14}, that is as in \cite[Section 4.4.3]{DiMarinoThesis}, and ${\rm lip}_a\, g$ is the asymptotic Lipschitz constant of a locally Lispchitz function $g$, which satisfies ${\rm lip}_a\, g \ge \lip g$. In the previous chain, the first equality follows from \cite[Theorem 1.1]{AmbrosioDiMarino14}, the second and third ones from \cite[Theorem 4.5.3]{DiMarinoThesis}, and the final inequality from the estimate ${\rm lip}_a \ge \lip$ for locally Lipschitz functions.

Hence \cite[Proposition 4.5.6]{DiMarinoThesis} applies, that yields the claim.
\end{proof}

The following general coarea formula will be employed in obtaining the sharp Sobolev inequality on $\CD(0, n)$ below.
\begin{remark}[Coarea formula on metric measure spaces]\label{rem:PerimeterMMS}
Let $(X,\dist,\meas)$ be a metric measure space. Let us observe that from the definitions given above, a Borel set $E$ with finite measure has finite perimeter if and only if the characteristic function $\chi_E$ belongs to $BV(X,\dist,\meas)$.

If $f \in BV(X,\dist,\meas)$, then $\{f>\alpha\}$ has finite perimeter for a.e. $\alpha \in \R$ and the \emph{coarea formula} holds
\begin{equation*}
    \int_X u \de |Df| = \int_{-\infty}^{+\infty} \left( \int_X u \de \abs{D\chi_{\{f>\alpha\}}} \right) \de \alpha,
\end{equation*}
for any Borel function $u:X\to [0,+\infty]$, see \cite[Proposition 4.2]{Miranda03}. If $f$ is also continuous and nonnegative, then $|Df|(\{f=\alpha\})=0$ for \emph{every} $\alpha \in [0,+\infty)$ and the \emph{localized coarea formula} holds
\begin{equation*}
    \int_{\{a<f<b\}} u \de |Df| = \int_a^b \left( \int_X u \de \abs{D\chi_{\{f>\alpha\}}} \right) \de \alpha,
\end{equation*}
for every Borel function $u:X\to [0,+\infty]$ and every $0\le a < b < +\infty$, see \cite[Corollary 1.9]{AmborsioBrueSemola19}.
\end{remark}

We recall a statement for the classical Bishop--Gromov volume and perimeter comparison.
The conclusions \eqref{1}, \eqref{2}, and the rigidity part of \cref{thm:BishopGromov} are consequences, e.g., of \cite[Theorem 3.101]{lafontaine}, \cite[Theorem 1.2 and Theorem 1.3]{SchoenYauLectures}, and the arguments within their proofs. The conclusion \eqref{3} follows from \cite[Corollary 2.22, item (i)]{PigolaRigoliSetti} and the coarea formula.

\begin{thm}[Bishop--Gromov comparison]
\label{thm:BishopGromov}
Let $(M^n,g)$ be a  complete Riemannian manifold such that $\ric\geq (n-1)K$ on $M^n$ in the sense of quadratic forms for some $K \in \R$. Let us set $T_K:=+\infty$ if $K\leq 0$, and $T_K:=\pi/\sqrt{K}$ if $K>0$. Then, for every $p\in M$ and for $r \leq T_K$ the following hold
    \begin{align}
    \label{1}
    &\frac{\vol(B_r(p))}{v(n,K,r)}\to 1\,\text{as $r\to 0$ and it is nonincreasing}, \\
    \label{2}
    &\frac{P(B_r(p))}{s(n, K, r)}\to 1 \,\text{as $r\to 0$ and it is almost everywhere nonincreasing}, \\
    \label{3}
    &\frac{P(B_r(p))}{s(n, K, r)} \leq \frac{\vol(B_r(p))}{v(n,K,r)}.
    \end{align}
    Moreover, if one has $\vol(B_{\overline r}(p))=v(n,K,\overline r)$ for some $\overline r\leq T_K$, then $B_{\overline r}(p)$ is isometric to the ball of radius $\overline r$ in the simply connected model of constant sectional curvature $K$ and dimension $n$.
\end{thm}

\begin{remark}[Bishop--Gromov comparison theorem on m.m.s.]\label{rem:PerimeterMMS2}
Let us recall that for an arbitrary $\CD((n-1)K,n)$ space $(X,\dist,\meas)$ the classical Bishop--Gromov volume comparison (cf. \cref{thm:BishopGromov}) still holds. More precisely, for a fixed $x\in X$, the function $\meas(B_r(x))/v(n,K,r)$ is nonincreasing in $r$ and the function $P(B_r(x))/s(n,K,r)$ is essentially nonincreasing in $r$, i.e., $P(B_R(x))/s(n,K,R) \le P(B_r(x))/s(n,K,r)$ for almost every radii $R\ge r$, see \cite[Theorem 18.8, Equation (18.8), Proof of Theorem 30.11]{VillaniBook}. Moreover, it holds that $P(B_r(p))/s(n, K, r)\leq \vol(B_r(p))/v(n,K,r)$ for any $r>0$, indeed the last inequality follows from the monotonicity of the volume and perimeter ratios together with the coarea formula on balls.

Moreover, if $(X,\dist,\mathcal{H}^n)$ is an $\RCD((n-1)K,n)$ space, one can conclude that $\mathcal{H}^n$-almost every point has a unique measure Gromov--Hausdorff tangent isometric to $\mathbb R^n$ (\cite[Theorem 1.12]{DePhilippisGigli18}), and thus, from the volume convergence in \cref{thm:volumeconvergence}, we get
 \begin{equation}\label{eqn:VolumeConv}
 \lim_{r\to 0}\frac{\mathcal{H}^n(B_r(x))}{v(n,K,r)}=\lim_{r\to 0}\frac{\mathcal{H}^n(B_r(x))}{\omega_nr^n}=1, \qquad \text{for $\mathcal{H}^n$-almost every $x$},
 \end{equation}
 where $\omega_n$ is the volume of the unit ball in $\mathbb R^n$. Moreover, since the density function $x\mapsto \lim_{r\to 0}\mathcal{H}^n(B_r(x))/\omega_nr^n$ is lower semicontinuous (\cite[Lemma 2.2]{DePhilippisGigli18}), the latter \eqref{eqn:VolumeConv} implies that the density is bounded above by the constant $1$. Hence, from the monotonicity at the beginning of the remark we deduce that, if $(X,\dist,\mathcal{H}^n)$ is an $\RCD((n-1)K,n)$ space, then for every $x\in X$ we have $\mathcal{H}^n(B_r(x))\leq v(n,K,r)$ for every $r>0$.
\end{remark}

Let us also recall a classical definition for the convenience of the reader.

\begin{defn}[$\mathrm{AVR}$ and Euclidean volume growth]
Let $(X,\dist,\meas)$ be an arbitrary $\CD(0,n)$ space with $n \in [1,+\infty)$. From \cref{rem:PerimeterMMS2}, the Bishop--Gromov monotonicity holds, and thus we can define, for an arbitrary $x\in X$, 
$$
\mathrm{AVR}(X,\dist,\meas):=\lim_{r\to +\infty}\frac{\meas(B_r(x))}{\omega_nr^n},
$$
the {\em asymptotic volume ratio} of $(X,\dist,\meas)$, where $\omega_n:=\pi^{n/2}/\Gamma(n/2+1)$. The previous limit is independent on $x\in M$. Notice that $\meas(B_r(x))\geq \mathrm{AVR}(X,\dist,\meas) \omega_nr^n$ for every $r>0$, and every $x\in X$. If $\mathrm{AVR}(X,\dist,\meas)>0$ we say that $(X,\dist,\meas)$ has {\em Euclidean volume growth}.
\end{defn}

Let us recall the basic definition of asymptotic cone in the setting of $\RCD$ spaces with Euclidean volume growth.

\begin{defn}[Asymptotic cones]\label{def:AsymptoticCone}
Let $(X,\dist,\meas,x)$ be a pointed $\RCD(0,n)$ space with $\mathrm{AVR}(X,\dist,\meas)>0$. For every sequence $\{r_i\}_{i\in\mathbb N}$ with $r_i\to+\infty$ the sequence of pointed metric measure spaces $\{(X,r_i^{-1}\dist,r_i^{-n}\meas,x)\}_{i\in\mathbb N}$ is precompact in the pmGH topology due to \cref{rem:GromovPrecompactness}. Every pmGH limit of such a sequence is a metric cone, by a slight modification of the proof of \cite[Proposition 2.8]{DePhilippisGigli18}. Any such limit is called an {\em asymptotic cone of $X$} and will be denoted by $C$ or by $C(Z)$ when we want to higlight the fact that $Z$ is the metric space that is the basis of the cone. We stress that every such $C$ is an $\RCD(0,N)$ space.

We refer to \cite[Definition 2.7]{DePhilippisGigli18} for the precise definition of metric cone. Notice that the class of asymptotic cones of $X$ is independent on the base point $x\in X$.
\end{defn}

\subsection{Finite perimeter sets and minimizing sequences for the isoperimetric problem}\label{sub:FiniteGH}

In this section we recall the main generalized existence result of \cite{AFP21}, which will be crucially used in this paper.
We need to recall a generalized $L^1$-notion of convergence for sets defined on a sequence of metric measure spaces converging in the pmGH sense. Such a definition is given in \cite[Definition 3.1]{AmborsioBrueSemola19}, and it is investigated in \cite{AmborsioBrueSemola19} capitalizing on the results in \cite{AmbrosioHonda17}.

\begin{defn}[$L^1$-strong and $L^1_{\mathrm{loc}}$ convergence]\label{def:L1strong}
Let $\{ (X_i, \dist_i, \mathfrak{m}_i, x_i) \}_{i\in \N}$  be a sequence of pointed metric measure spaces converging in the pmGH sense to a pointed metric measure space $ (Y, \dist_Y, \mu, y)$ and let $(Z,\dist_Z)$ be a realization as in \cref{def:GHconvergence}.

We say that a sequence of Borel sets $E_i\subset X_i$ such that $\mathfrak{m}_i(E_i) < +\infty$ for any $i \in \N$ converges \emph{in the $L^1$-strong sense} to a Borel set $F\subset Y$ with $\mu(F) < +\infty$ if $\mathfrak{m}_i(E_i) \to \mu(F)$ and $\chi_{E_i}\mathfrak{m}_i \rightharpoonup \chi_F\mu$ with respect to the duality with continuous bounded functions with bounded support on $Z$.

We say that a sequence of Borel sets $E_i\subset X_i$ converges \emph{in the $L^1_{\mathrm{loc}}$-sense} to a Borel set $F\subset Y$ if $E_i\cap B_R(x_i)$ converges to $F\cap B_R(y)$ in $L^1$-strong for every $R>0$.
\end{defn}

Observe that in the above definition it makes sense to speak about the convergence $\chi_{E_i}\mathfrak{m}_i \rightharpoonup \chi_F\mu$ with respect to the duality with continuous bounded functions with bounded support on $Z$ as $(X_i,\dist_i),(Y,\dist_Y)$ can be assumed to be topological subspaces of $(Z,\dist_Z)$ by means of the isometries $\Psi_i,\Psi$ of \cref{def:GHconvergence}, and the measures $\mathfrak{m}_i,\mu$ can be then identified with the push-forwards $(\Psi_i)_\sharp\mathfrak{m}_i,\Psi_\sharp\mu$ respectively.

Let us recall here for the reader's convience the main result of \cite{AFP21}, namely \cite[Theorem 1.1]{AFP21}. We will crucially need this result for the proof of our main results.
\begin{thm}[{Asymptotic mass decomposition \cite[Theorem 1.1]{AFP21}}]\label{thm:MassDecomposition}
Let $(M^n,g)$ be a noncollapsed noncompact complete manifold with infinite volume, such that $\ric\ge K$ for some $K\in(-\infty, 0]$, and let $V>0$. For every minimizing (for the perimeter) sequence of sets $\Omega_i\subset M^n$ of volume $V$, with $\Omega_i$ bounded for any $i$, up to passing to a subsequence, there exist an increasing sequence $\{N_i\}_{i\in\mathbb N}\subseteq \mathbb N$, disjoint finite perimeter sets $\Omega_i^c, \Omega_{i,j}^d \subset \Omega_i$, and points $p_{i,j}$, with $1\leq j\leq N_i$ for any $i$, such that
\begin{itemize}
    \item[(i)] $\lim_{i} \dist(p_{i,j},p_{i,\ell}) = \lim_{i} \dist(p_{i,j},o)=+\infty$, for any $j\neq \ell<\overline N+1$ and any $o\in M^n$, where $\overline N:=\lim_i N_i \in \N \cup \{+\infty\}$;
    
    \item[(ii)] $\Omega_i^c$ converges to $\Omega\subset M^n$ in the sense of finite perimeter sets (\cref{def:ConvergenceFinitePerimeter}), and we have $\vol(\Omega_i^c)\to_i \vol(\Omega)$, and $ P( \Omega_i^c) \to_i P(\Omega)$. Moreover $\Omega$ is a bounded isoperimetric region on $M$;
    
    \item[(iii)] for every $j<\overline N+1$, $(M^n,\dist,\vol,p_{i,j})$ converges in the pmGH sense  to a pointed $\RCD(K,n)$ space $(X_j,\dist_j,\meas_j,p_j)$, where $\meas_j$ is the $n$-dimensional Hausdorff measure on $(X_j,\dist_j)$. Moreover there are isoperimetric regions $Z_j \subset X_j$ such that $\Omega^d_{i,j}\to_i Z_j$ in $L^1$-strong (\cref{def:L1strong}) and $P(\Omega^d_{i,j}) \to_i P_{X_j}(Z_j)$;
    
    \item[(iv)] it holds that
    \begin{equation}\label{eq:UguaglianzeIntro}
    I_{(M^n,g)}(V) = P(\Omega) + \sum_{j=1}^{\overline{N}} P_{X_j} (Z_j),
    \qquad\qquad
    V=\vol(\Omega) +  \sum_{j=1}^{\overline{N}} \mathfrak{m}_j(Z_j).
    \end{equation}
\end{itemize}
\end{thm}

We remark that item (ii) in \cref{thm:MassDecomposition} is, in fact, proved in \cite[Theorem 2.1]{RitRosales04}, and it consists in the starting point for the rest of the proof of \cref{thm:MassDecomposition}.

We will need the following result on the boundedness of isoperimetric regions.

\begin{prop}[{\cite[Corollary 4.2]{AFP21}}]\label{cor:IsopBounded}
Let $(M^n,g)$ be a complete noncollapsed Riemannian manifold with $\ric \geq K$ for some $K \in (- \infty, 0]$. Then the isoperimetric regions of $(M^n,g)$ are bounded.
\end{prop}

\subsection{The collapsed case}

As already observed in \cite[Remark 4.7]{AFP21}, as a nontrivial consequence of \cref{thm:MassDecomposition} we have that  complete noncollapsed manifolds with a lower bound on the Ricci curvature have strictly positive isoperimetric profile for any volume.

With an argument partly inspired by the proof of \cite[Proposition 3.13]{LeonardiRitore}, we are now going to show that, in the nonnegative Ricci case, collapsedness occurs if and only if the isoperimetric profile vanishes for any volume. In particular, no isoperimetric sets exist in this situation. 

\begin{prop}\label{lem:NonSoDoveMetterlo}
Let $(M^n,g)$ be a complete noncompact Riemannian manifold with $\ric\geq 0$. Then, $\inf_{x \in M} \vol({B_1(x)}) = 0$ if and only if $I(V) = 0$ for any positive volume $V$.
\end{prop}

\begin{proof}
As observed above, if $M$ is noncollapsed, then $I(V)>0$ for any $V$ by \cite[Remark 4.7]{AFP21}. Then let us assume that $\inf_{x \in M} \vol(B_1(x))=0$.
Observe first that $\inf_{x \in M} \vol (B_R(x)) = 0$ for any $R > 0$. Indeed, let $\{x_j\}_{j \in \N}$ be a sequence of points such that $\lim_{j \to \infty} \vol B_1(x_j) = 0$. 
Now, for $R \leq 1$ we have $\vol B_R(x_j) \leq \vol(B_1(x_j))\to_j 0$, while for $R>1$, by the Bishop--Gromov volume comparison, we have $\vol (B_R(x_j)) \leq R^n \vol (B_1(x_j))\to_j 0$. Fix now a volume $V > 0$, and consider an arbitrary ball $B_R(x)$ such that $\vol (B_R(x)) > V$. Since, as we just showed, $\vol (B_R(x_j)) \to 0$ along the sequence $x_j$ above we can in particular find a point $\overline{x} \in M$ such that $\vol (B_R(\overline{x})) = V$.  By \eqref{3}, we have $ P(B_R(\overline{x})) \leq n R^{-1} \vol (B_R(\overline{x}))$, and thus 
\begin{equation}
    \label{Ivaa0}
    I(V) \leq P(B_R(\overline{x}))\leq  \frac{n}{R} V.
\end{equation}
Observing that the argument actually holds true for arbitrarily big radii $R$, our claim is proved by letting $R \to \infty$ in \eqref{Ivaa0}.
\end{proof}
Explicit examples of collapsed manifolds with nonnegative Ricci curvature have been constructed in dimension $n \geq  4$ by Croke and Karcher in \cite[Example 1 and Example 2]{CrokeKarcher}. Hence for such examples no isoperimetric regions of positive volume can exist.

In the same paper, it is shown that the latter examples do not exist in dimension $n = 2$. As a consequence of \cref{lem:NonSoDoveMetterlo} and a former result of Ritoré \cite[Theorem 1.1]{RitoreExistenceSurfaces01}, asserting that surfaces with nonnegative curvature admit isoperimetric regions of any given volume, we are able to fully recover such a positive result proved in \cite[Theorem A]{CrokeKarcher}. We remark that the proof given here is completely different from the one given in \cite{CrokeKarcher} and recovers the geometric conclusion of the statement passing through the isoperimetric problem.

\begin{cor}
Let $(M^2,g)$ be a noncompact complete Riemannian surface with $\sect \geq 0$. Hence there exists $C:=C(M)$ such that 
$$
\vol(B(x,1))\geq C, \qquad \forall x\in M.
$$
\end{cor}
\begin{proof}
If the conclusion does not hold, from \cref{lem:NonSoDoveMetterlo} we get that $I_{(M^2,g)}(V)=0$ for every $V>0$ and hence no isoperimetric regions of positive volume can exist on $(M^2,g)$. But this is in contradiction with \cite[Theorem 1.1 \& Theorem 2.1]{RitoreExistenceSurfaces01}.
\end{proof}

\begin{remark}
The proof of the previous result \cref{lem:NonSoDoveMetterlo} builds on uses the Bishop--Gromov monotonicity results, which hold for $\CD(0,n)$ spaces, cf. \cref{rem:PerimeterMMS2}. In fact, one can show with the same argument that also on a $\CD(0,n)$ space $(X,\dist,\meas)$, if $\inf_{x\in X}\meas(B_1(x))=0$, then the isoperimetric profile identically vanishes.
\end{remark}

\subsection{Sharp Sobolev inequality on $\CD(0,n)$ spaces}
Here, we improve the sharp isoperimetric inequality on $\CD(0, n)$ recently obtained in \cite{BaloghKristaly}. Namely, while such result is written in terms of the Minkowski content of the set involved, we provide one, with the same sharp constant, in terms of the perimeter. In doing so, we are actually getting also a sharp Sobolev inequality in these spaces. Earlier versions of inequalities like these have been obtained in \cite{AgostinianiFogagnoloMazzieri, BrendleFigo, FogagnoloMazzieri}. 

\begin{thm}[Sharp Sobolev inequality on $\CD(0,n)$ spaces]\label{thm:IsoperimetriSharpBrendle}
Let $(X,\dist,\meas)$ be a $\CD(0,n)$ space. Then for any $f\in \text{BV}(X,\dist,\meas)$ it holds
\begin{equation}\label{eq:Sobolevsharp}
    n \omega_n^{\frac{1}{n}} \mathrm{AVR}(X,\dist,\meas)^{\frac{1}{n}} \left( \int_X |f|^{\frac{n}{n-1}} \de \meas   \right)^{\frac{n-1}{n}}
    \le 
    |D f|(X) \, .
\end{equation}
In particular, for any set of finite perimeter $E$, with $\meas(E)<\infty$, we have
\begin{equation}\label{eq:isoperimetricCD}
    P_X(E)\geq n\omega_n^{\frac{1}{n}}\mathrm{AVR}(X,\dist,\meas)^{\frac{1}{n}}\meas(E)^{\frac{n-1}{n}}.
\end{equation}
\end{thm}

The isoperimetric inequality \cite[Theorem 1.1]{BaloghKristaly} reads, for any bounded subset $E\subset X$ of a $\CD(0,n)$ space $(X,\dist,\meas)$,
\begin{equation}\label{eq:Balogh-Kristaly}
    \meas^+(E) \ge  n\omega_n^{\frac{1}{n}} \mathrm{AVR}(X,\dist,\meas)^{\frac{1}{n}} \meas(E)^{\frac{n-1}{n}} \, ,
\end{equation}
where $\meas^+$ denotes the (lower) Minkowski content
\begin{equation}
    \meas^+(E):= \liminf_{r\to 0} \frac{\meas(B_r(E))-\meas(E)}{r}\, .
\end{equation}
The proof of \cref{thm:IsoperimetriSharpBrendle} substantially follows from \eqref{eq:Balogh-Kristaly} together with a relaxation argument \cite{AmbrosioGigliDimarino}. Observe that \eqref{eq:isoperimetricCD} also allows for unbounded sets.

\smallskip

We will need  the following elementary lemma proven in  \cite[Lemma A.24]{AmbrosioPDE}.
\begin{lemma}\label{lemma:elementary}
Let $G:[0,\infty)\to [0,\infty)$ a nonincreasing measurable function. Then for any $\alpha \ge 1$ we have
\begin{equation}
    \alpha \int_0^\infty t^{\alpha-1} G(t) \de t 
    \le
    \left( \int_0^{\infty} G^{1/\alpha}(t) \de t   \right)^{\alpha} \, .
\end{equation}
\end{lemma}

\begin{proof}[Proof of \cref{thm:IsoperimetriSharpBrendle}]

Let $f:X\to [0,\infty)$ be a Lipschitz function with bounded support. 
From \cite[Proposition 4.2]{AmbrosioGigliDimarino} we know that
\begin{equation*}
    \int_X \lip(f) \de \meas = \int_0^{\infty} \meas^+(\{f>t\}) \de t,
\end{equation*}
hence, from \eqref{eq:Balogh-Kristaly} we deduce
\begin{align*}
  \int_X \lip(f) \de \meas \,
  =
  \int_0^\infty\meas^+(\{f>t\}) \de t
  \, \ge
  n \omega_n^{\frac{1}{n}} \mathrm{AVR}(X,\dist,\meas)^{\frac{1}{n}} \int_0^\infty \meas(\{f>t\})^{\frac{n-1}{n}} \de t.
\end{align*}
A simple application of \cref{lemma:elementary} coupled with the coarea formula in \cref{rem:PerimeterMMS} gives
\begin{equation*}
    \int_0^\infty \meas(\{f>t\})^{\frac{n-1}{n}} \de t
    \ge
    \left( \int_X f^{\frac{n}{n-1}} \de\meas   \right)^{\frac{n-1}{n}}\, ,
\end{equation*}
for any nonnegative Lipschitz function with bounded support.
We can easily drop the assumption on the sign of $f$ by using the decomposition in positive and negative part. Let us finally
show \eqref{eq:Sobolevsharp} via approximation argument. Let us fix $f\in BV(X,\dist,\meas)$, and approximate it with $(f_k)$ as in \cref{lemma:approximationSobolevfunctio}. For any $k\in \mathbb{N}$ it holds 
\begin{equation*}
    \int_X \lip(f_k) \de\meas \ge n \omega_n^{\frac{1}{n}} \mathrm{AVR}(X,\dist,\meas)^{\frac{1}{n}}
    \left( \int_X |f_k|^{\frac{n}{n-1}} \de \meas   \right)^{\frac{n-1}{n}} \, ,
\end{equation*}
and passing to the limit as $k\to \infty$ we deduce
\begin{align*}
    |Df|(X)  & \ge \liminf_{k\to \infty} n \omega_n^{\frac{1}{n}} \mathrm{AVR}(X,\dist,\meas)^{\frac{1}{n}}
    \left( \int_X |f_k|^{\frac{n}{n-1}} \de \meas   \right)^{\frac{n-1}{n}} 
    \\ & \ge \omega_n^{\frac{1}{n}} \mathrm{AVR}(X,\dist,\meas)^{\frac{1}{n}}
    \left( \int_X |f|^{\frac{n}{n-1}} \de \meas   \right)^{\frac{n-1}{n}} \, 
\end{align*}
where the last inequality follows from Fatou's lemma.

The last conclusion \eqref{eq:isoperimetricCD} comes by plugging $f=\chi_E$ in \eqref{eq:Sobolevsharp}.
\end{proof}

\section{Concavity properties of the isoperimetric profile}\label{sec:Concavity}

In this section we study concavity properties of the isoperimetric profile of a noncollapsed Riemannian manifold with Ricci bounded from below. We then apply the obtained results to the case of manifolds with nonnegative Ricci curvature and Euclidean volume growth. In the first part of this section we consider a {\em penalized} isoperimetric problem and we show that its associated isoperimetric profile converges locally uniformly to the isoperimetric profile of the manifold.

Throughout this section let $(M^n,g)$ be a fixed complete Riemannian manifold, and let $o \in M$.
Fix a sequence $\eps_k\searrow 0$ and for any $k\in\mathbb N$ let $U_k$ be a bounded open smooth set in $M$ containing the ball $B_{1/\eps_k}(o)$ such that $U_k\subset U_{k+1}$; and choose $f_k\in C^\infty(M)$ such that
\[
f_k\ge0 ,
\qquad
f_k|_{U_k} \equiv 0, 
\qquad
|\nabla f_k|\le \eps_k,
\qquad
f_{k+1}\le f_k,
\]
for any $k\in \N$. The previous choice can be made in such a way that for any $k$ and for any $N>0$ there is $R>0$ such that $f_k>N$ on $M\setminus B_R(o)$. Define, for any $k\in\mathbb N$,
\begin{equation}\label{eq:DefPk}
P_k(\Omega) \eqdef P(\Omega) + \int_\Omega f_k \de \vol,
\end{equation}
for any set of finite perimeter $\Omega\subset M$, and
\begin{equation}\label{eq:DefIk}
I_k(V) \eqdef \inf\{ P_k(\Omega) \st \text{$\Omega$ is a finite perimeter set with}\, \vol(\Omega) = V\},
\end{equation}
for any $V>0$. For simplicity, we shall denote by $I$ the isoperimetric profile $I_{(M^n,g)}$. 

\begin{lemma}\label{lem:ApprossimazioneProfili}
Suppose that $(M^n,g)$ is noncollapsed and that $\ric \ge (n-1)K$ for some $K \in \R$. Then the following two facts hold.
\begin{enumerate}
    \item For any $k\in \N$ and any $V>0$ there exists $\Omega$ of volume $V$ such that $I_k(V) = P_k(\Omega)$.
    
    \item For any $k\in\N$ the function $I_k:(0,+\infty)\to \R$ is continuous and the sequence $\{I_k\}_{k\in \N}$ converges from above to the isoperimetric profile $I_{(M^n,g)}$ locally uniformly on $(0,+\infty)$.
\end{enumerate}
\end{lemma}

\begin{proof}
For any $k$ and $V>0$ the existence of a minimizer $\Omega$ for $P_k$ at volume $V$ follows by means of the direct method. Indeed a minimizing sequence $\{\Omega_i\}_{i\in\mathbb N}$ for $P_k$ at volume $V=\vol(\Omega_i)$ converges up to subsequence to some set $\Omega$ in $L^1_{\rm loc}(M)$. The functional $P_k$ is lower semicontinuous with respect to $L^1_{\mathrm{loc}}$-convergence and if by contradiction $\vol(\Omega)<V$, since $f_k$ diverges at infinity, we would get that $\int_{\Omega_i}f_k \to +\infty$ as $i\to+\infty$.

For any $k$ and $V>0$ denote by $\Omega^k_V$ a minimizer of $P_k$ at volume $V$. First observe that $I_k$ is locally bounded; indeed for any $V>0$ the value of $I_k$ at volumes $v$ in a neighborhood of $V$ is estimated from above by the value of $P_k$ on balls of center $o$ and given volume $v$.

Let us now prove that $I_k$ is continuous. Indeed, fix $V>0$ and let $V_i\to V$ be any sequence. Since $I_k(V_i) = P_k(\Omega_{V_i}^k)$, and thus the perimeters of $\Omega_{V_i}^k$ are uniformly bounded in $i$, we get that, up to subsequence, $\{\Omega_{V_i}^k\}_{i\in\mathbb N}$ converges in $L^1_{\rm loc}$ to some set $\Omega$. As $\{P_k(\Omega_{V_i}^k)\}_{i\in\mathbb N}$ is uniformly bounded, we deduce that $\vol(\Omega)=V$ and then the convergence $\Omega_{V_i}^k\to_i \Omega$ holds in $L^1$ on $M$.
Moreover
\[
I_k(V) \le P_k(\Omega) \le \liminf_i P_k(\Omega_{V_i}^k) = \liminf_i I_k(V_i).
\]
Hence $I_k$ is lower semicontinuous. Let us now finish the proof of the continuity. Indeed,
suppose by contradiction that $I_k(V) \le \limsup_i I_k(V_i) - \delta$ for some $\delta>0$. There exists a set $E$ of volume $V$ with $P_k(E)\le I_k(V) + \delta/4$, and then for large $i$ we can find a ball $B_{r_i}(x_i)$ such that $r_i\to0^+$ and either $E \cup B_{r_i}(x_i)$ or $E \setminus B_{r_i}(x_i)$ has volume $V_i$, and, denoting by $E_i$ such a set of volume $V_i$, we have
\begin{equation}\label{eq:ContinuityIk}
P_k(E_i)  \le I_k(V) + \delta/2,
\end{equation}
for large $i$. For the last inequality to hold we are using that, as a consequence of the Bishop--Gromov comparison, $P(E\cup B_{r_i}(x_i))\leq P(E)+v(n,K,r_i)$, together with the fact that $r_i\to 0$. Finally, $P_k(E_i) \ge I_k(V_i)$, thus inserting the absurd hypothesis in \eqref{eq:ContinuityIk} and passing to the limsup in $i$ we derive a contradiction.

Therefore $\lim_i I_k(V_i) = I_k(V)$, and the arbitrariness of the sequence implies that $I_k$ is continuous at $V$, and then $I_k$ is continuous on $(0,+\infty)$.

Since $(M^n,g)$ is noncollapsed with Ricci bounded below, the isoperimetric profile $I$ is continuous and strictly positive by \cite[Corollary 2]{FloresNardulli20} and \cite[Remark 4.7]{AFP21}. Since $I_k$ is continuous and $I\le I_{k+1}\le I_k$ for any $k\in \N$, if $I_k\to_k I$ pointwise, then the same convergence holds locally uniformly by Dini's Monotone Convergence Theorem. Now for fixed $V>0$ we show, indeed, that $I_k(V)\to_k I(V)$. If by contradiction $\limsup_k I_k(V) \ge \delta + I(V)$ for some $\delta>0$, then one can consider a bounded smooth set $\Omega$ of volume $V$ such that $P(\Omega) \le I(V) + \delta/2$. Since $\Omega \Subset U_k$ for large $k$, we get that
\[
I(V) + \frac\delta2 \ge P(\Omega) = \limsup_k P_k(\Omega) \ge \delta + I(V),
\]
that is impossible.
\end{proof}

\subsection{First and second variations of the penalized perimeter}

In this section we derive first and second variations for the potential term in $P_k$, and we collect basic properties of the sets that minimize the penalized perimeter $P_k$ under a volume constraint.

\begin{lemma}\label{lem:VariazioniPotenziale}
Let $M$ be a complete Riemannian manifold. Let $\Omega \subset M$ be a set of finite perimeter and suppose that there is a smooth hypersurface $\Sigma$ with boundary contained in the essential boundary $\partial^*\Omega$. Let $X:\Sigma \to TM$ be a smooth vector field compactly supported in the interior of $\Sigma$ and denote $\phi_t(x)\eqdef\exp(tX(x))$ for $x \in \Sigma$.

For some $\tau>0$, if $\Omega_t$ is the set of finite perimeter with essential boundary $\phi_t(\partial^*\Omega)$ for $|t|<\tau$ and $X_t\eqdef \partial_t \phi_t$ is the variation field along $\phi_t(\Sigma)$, then
\begin{equation}\label{eq:FirstVariationPotenziale}
    \frac{\d}{\d t} \left( \int_{\Omega_t} f_k \de \vol \right) = \int_{\partial^* \Omega_t} f_k \scal{ X_t , - \nu_{\Omega_t} } \de \mathcal{H}^{n-1}
    \qquad \forall\,|t|<\tau,
\end{equation}
where $\nu_{\Omega_t}$ is the generalized interior unit normal of $\Omega_t$. If also $X$ is normal along $\Sigma$, then
\begin{equation}\label{eq:SecondVariationPotenziale}
\begin{split}
        \frac{\d^2}{\d t^2} \left( \int_{\Omega_t} f_k \de \vol \right)\bigg|_{t=0} = & \int_\Sigma \left(f_k H_\Omega \scal{X,\nu_\Omega}^2 - \scal{\nabla f_k , X}\scal{X, \nu_\Omega}\right)  \de \mathcal{H}^{n-1},
\end{split}
\end{equation}
where $\nu_{\Omega}$ is the generalized interior unit normal of $\Omega$ and $H_\Omega$ denotes the mean curvature in direction $\nu_\Omega$, i.e., $H_\Omega\nu_\Omega$ is the mean curvature vector of $\Sigma$.
\end{lemma}

\begin{proof}
The first variation \eqref{eq:FirstVariationPotenziale} follows by an application of the area formula as in \cite[Proposition 17.8]{MaggiBook}.
The second variation \eqref{eq:SecondVariationPotenziale} follows by differentiating \eqref{eq:FirstVariationPotenziale}. Indeed, rewriting the varied boundary $\partial^*\Omega_t$ as an embedding $F_t:\Sigma\to M$, where $F_t(\Sigma)= \phi_t(\Sigma)$, we have
\[
\frac{\d}{\d t} \left( \int_{\Omega_t} f_k \de \vol \right) = \int_{\Sigma} [f_k \scal{ X_t , - \nu_{\Omega_t} }]\circ F_t \de \vol(t),
\]
where $\vol(t)$ is the volume form on $\Sigma$ induced by $F_t$. Hence
\begin{equation*}
\begin{split}
    \frac{\d^2}{\d t^2} \left( \int_{\Omega_t} f_k \de \vol \right)= & \int_\Sigma - \scal{\nabla f_k , X_t}\scal{X_t, \nu_{\Omega_t}} \circ F_t - f_k \scal{ \nabla_{X_t} X_t, \nu_{\Omega_t}} \circ F_t +\\ 
    &- f_k \scal{X_t, \nabla_{X_t} \nu_{\Omega_t}} \de \vol(t)  + \int_\Sigma f_k\scal{X_t,-\nu_{\Omega_t}} \circ F_t \, \partial_t \de\vol(t) .
\end{split}
\end{equation*}
Since $\nabla_{X_t}X_t=0$ and $\nabla_{X_t}\nu_{\Omega_t}$ is tangent along $F_t(\Sigma)$, evaluating at $t=0$ we get
\begin{equation}\label{eq:SecVar1}
\begin{split}
    \frac{\d^2}{\d t^2} \left( \int_{\Omega_t} f_k \de \vol \right)\bigg|_{t=0}= & \int_\Sigma - \scal{\nabla f_k , X}\scal{X, \nu_{\Omega}}\de\mathcal{H}^{n-1}   + \int_\Sigma f_k\scal{X,-\nu_{\Omega}} \, \partial_t \de\vol(t)|_{t=0} .
\end{split}
\end{equation}
Since $X$ is normal along $\Sigma$, i.e., its tangent projection $X^\top$ onto $T\Sigma$ vanishes, the identity \eqref{eq:SecondVariationPotenziale} follows by the fact that
\begin{equation}\label{eq:SecVar2}
\partial_t \de\vol(t)|_{t=0} = \left(\div_{\Sigma} ( X^\top) - H_\Omega \scal{X, \nu_\Omega}\right)\de\mathcal{H}^{n-1}\llcorner\Sigma =   - H_\Omega \scal{X, \nu_\Omega}\de\mathcal{H}^{n-1}\llcorner\Sigma .
\end{equation}
\end{proof}

In the hypotheses and notation of \cref{lem:VariazioniPotenziale}, if $\varphi$ is a smooth function compactly supported in $\Sigma\subset \partial^*\Omega$, we can plug in the normal field $X= -\varphi\nu_\Omega$, so that \eqref{eq:SecondVariationPotenziale} reads
\begin{equation}\label{eq:SecondVariationPotenzialeSpecial}
        \frac{\d^2}{\d t^2} \left( \int_{\Omega_t} f_k \de \vol \right)\bigg|_{t=0} =  \int_{\partial^*\Omega} \left( \varphi^2 f_k H_\Omega - \varphi^2 \scal{\nabla f_k , \nu_\Omega}\right) \de \mathcal{H}^{n-1}.
\end{equation}

\begin{lemma}\label{lem:MinimizingPk}
Suppose that $(M^n,g)$ is noncollapsed and that $\ric \ge (n-1)K$ for $K \in \R$.
Let $\Omega$ be a minimizer of $P_k$, see \eqref{eq:DefPk}, at volume $V$. Then $\Omega$ is bounded and there exists a constant $\lambda=\lambda_{k,V}\in\R$ such that
\[
\frac{\d}{\d t} P_k(\phi_t(\Omega))\bigg|_{t=0} = \lambda \int_{\partial^*\Omega} \scal{X, - \nu_\Omega} \de \mathcal{H}^{n-1},
\]
for any smooth compactly supported vector field $X$ on $M$, where $\phi_t$ is the flow of $X$ and $\nu_\Omega$ the generalized interior normal of $\Omega$.

In particular $\Omega$ has generalized mean curvature $H_\Omega:\partial^*\Omega\to \R$ such that
\begin{equation}\label{eq:Lambda}
    H_\Omega + f_k = \lambda
    \qquad
    \text{$\mathcal{H}^{n-1}$-a.e. on $\partial^*\Omega$,}
\end{equation}
that is
\[
\int_{\partial^*\Omega} (H_\Omega + f_k) \scal{X, - \nu_\Omega} \de \mathcal{H}^{n-1}
= \lambda \int_{\partial^*\Omega} \scal{X, - \nu_\Omega} \de \mathcal{H}^{n-1},
\]
for any smooth vector field $X$ on $M$.

Moreover there is an open representative of $\Omega$ and its topological boundary $\partial\Omega$ is given by the disjoint union $\partial\Omega = \partial_r\Omega \sqcup \partial_s\Omega$, where $\partial_r\Omega$ is a smooth hypersurface with topological boundary $\partial_s\Omega$ that is empty if the dimension is $n \leq 7$ and with Hausdorff dimension less than or equal to $n-8$ when $n \geq 8$.
\end{lemma}

\begin{proof}
The first variation $\partial_t P_k(\phi_t(\Omega))|_{t=0}$ vanishes whenever $X$ generates a volume preserving variation, i.e., whenever $\int_{\partial^*\Omega} \scal{X,\nu_\Omega} =0$. Arguing for example as in \cite[Theorem 17.20]{MaggiBook} it is a classical matter to derive the existence of the Lagrange multiplier $\lambda$.

Being $\Omega, X, \phi_t$ are as in the statement, then from \eqref{eq:FirstVariationPotenziale} we deduce that
\[
\frac{\d}{\d t} \left( \int_{\Omega_t} f_k \de \vol \right) \bigg|_{t=0} = \int_{\partial^* \Omega} f_k \scal{ X , - \nu_{\Omega} } \de \mathcal{H}^{n-1},
\]
where $\Omega_t\eqdef\phi_t(\Omega)$. The first variation formula for the perimeter then yields that
\[
\frac{\d}{\d t} P_k(\phi_t(\Omega))\bigg|_{t=0} = \int_{\partial^*\Omega} \left(\div X - \scal{\nabla X (\nu_\Omega), \nu_\Omega } - f_k \scal{X,\nu_\Omega}\right)\de\mathcal{H}^{n-1}.
\]
Letting $\div_{\Omega} X \eqdef \div X - \scal{\nabla X (\nu_\Omega), \nu_\Omega }$ the tangential divergence of $X$ along $\partial^*\Omega$, we deduce that
\[
\int_{\partial^*\Omega} \div_{\Omega} X \de\mathcal{H}^{n-1}
= \int_{\partial^*\Omega} (f_k - \lambda )\scal{X,\nu_\Omega}\de\mathcal{H}^{n-1},
\]
which precisely defines the generalized mean curvature of $\Omega$ via the identity $H_\Omega\nu_\Omega = - (f_k-\lambda) \nu_\Omega$.

The regularity of $\partial \Omega$ in the statement follows by \cite[Corollary 1.6 \& Remark 1.7]{AntonelliPasqualettoPozzetta21}. These results imply that $\Omega$ has an open representative with topological boundary $\partial\Omega$ given by the disjoint union $\partial\Omega = \partial_r\Omega \sqcup \partial_s\Omega$, where $\partial_r\Omega$ is locally $C^{1,\alpha}$ and it has topological boundary $\partial_s\Omega$ with Hausdorff dimension less than or equal to $n-8$. Since $f_k$ is smooth, the Euler--Lagrange equation \eqref{eq:Lambda} classically implies that the regular part $\partial_r\Omega$ is, in fact, smooth.

Now we can prove that $\Omega$ is bounded by a modification of a classical argument appeared in \cite[Proposition 3.7]{RitRosales04} in the Euclidean setting and in \cite[Theorem 3]{Nar14} on Riemannian manifolds (see also \cite[Appendix B]{AFP21}). Fix $p_0 \in \partial_r \Omega$. Since $\partial_s\Omega$ is closed, there is $R_1>0$ such that $\partial\Omega \cap B_{R_1}(p_0)$ is a nonempty smooth hypersurface, and it is well defined the inner normal $\nu_\Omega$ of $\Omega$ on $\partial\Omega \cap B_{R_1}(p_0)$. Let $\varphi \in C^\infty_c(\partial\Omega \cap B_{R_1}(p_0))$ be a nonvanishing nonnegative function and consider the normal vector field $X=-\varphi \nu_\Omega$. Let $\Phi_t(x)\eqdef\exp(tX(x))$ and let $\Omega_t$ be the varied set whose essential boundary is $\Phi_t(\partial^*\Omega)$, for $|t|<\tau$ and $\tau>0$. Since $\alpha\eqdef \int_{\partial\Omega} \scal{X,-\nu_\Omega} >0$, first variation formulae for perimeter and volume give that
\begin{equation}\label{eq:Variazioni1}
    \begin{split}
        \vol(\Omega_t) &= \vol(\Omega) + \alpha t + O(t^2),\\
    P(\Omega_t,B_{R_1}(p_0)) &= P(\Omega,B_{R_1}(p_0)) + t \int_{\partial\Omega} \varphi H_\Omega\de\mathcal{H}^{n-1}  + O(t^2),
    \end{split}
\end{equation}
for $|t|<\tau$. Moreover, as $\varphi>0$, we have that
\begin{equation}\label{eq:Contenimento}
    \Omega \subset \Omega_t \qquad \text{for $0\le t<\tau$,}
\end{equation}
up to decreasing $\tau$. As $\alpha>0$ and $|\int \varphi H_\Omega|<+\infty$, \eqref{eq:Variazioni1} and \eqref{eq:Contenimento} imply that there exists $\eps_0>0$ and $\beta>0$ such that for any $\eps\in (-\eps_0,\eps_0)$ there is a finite perimeter set $E$ with
\begin{equation}\label{eq:LocalModificationsSets}
\begin{split}
    &E\Delta \Omega \Subset B_{R_1}(p_0),\\
    &    \vol(E)=\vol(\Omega)+\eps,\\
    &P(E,B_{R_1}(p_0)) \le P(\Omega,B_{R_1}(p_0))+ \beta|\eps|, \\
    & \eps>0 \text{ implies } \Omega \subset E.
\end{split}
\end{equation}
Now for $r>R_1$ let
\[
V(r)\eqdef \vol(\Omega \setminus B_r(p_0)), 
\qquad
A(r)\eqdef P(\Omega, M\setminus B_r(p_0)).
\]
Since $M$ is noncollapsed and the Ricci curvature is bounded from below, it holds an isoperimetric inequality for sufficiently small volumes, see \cite[Lemma 3.2]{Heb00}. Hence for some $R_2>R_1$, for $r\ge R_2>0$ we can apply such an isoperimetric inequality on $\Omega \setminus B_r(p_0)$, so that
\[
| V'(r) | + A(r) = \mathcal{H}^{n-1}(\partial B_r(p_0) \cap \Omega) + P(\Omega, M\setminus B_r(p_0)) = P(\Omega\setminus B_r(p_0)) \ge c_0 V(r)^{\frac{n-1}{n}},
\]
for some $c_0>0$ for almost every $r\ge R_2$.
We want to prove that
\begin{equation}\label{eq:GoalBddIsopRegion}
    A(r) \le | V'(r) |  + C V(r) ,
\end{equation}
for some constant $C$, and for almost every $r$ sufficiently big. Combining  with the previous inequality, in this way we would get
\[
c_0 V(r)^{\frac{n-1}{n}} \le  C V(r)  + 2| V'(r) | \le \frac{c_0}{2}V(r)^{\frac{n-1}{n}} - 2 V'(r),
\]
because $| V'(r) | = - V'(r) $ and $ C V(r)\le \tfrac{c_0}{2}V(r)^{\frac{n-1}{n}} $ for almost every sufficiently big radius. Hence ODE comparison implies that $V(r)$ vanishes at some $r=\overline{r}<+\infty$, i.e., $\Omega$ is bounded as a set of finite perimeter.

So we are left to prove \eqref{eq:GoalBddIsopRegion}. For some $R_3\ge R_2$, for $r\ge R_3$ we can assume that $V(r)<\eps_0$. For fixed $r\ge R_3$ let $\eps\eqdef V(r)$ and assume also that  $\mathcal{H}^{n-1}(\partial\Omega \cap \partial B_r(p_0))=0$ and $V'(r)$ exists, which hold for almost every $r$. Given such $\eps$, there exists $E$ as in \eqref{eq:LocalModificationsSets}. Finally let $F= E \cap B_r(p_0)$. It follows that $\vol(F) = \vol(\Omega)$ and \eqref{eq:LocalModificationsSets} implies that
\begin{equation}\label{eq:VolumeF}
    \Omega \cap B_r(p_0) \subset F,
    \qquad
    F\setminus\Omega \subset B_{R_1}(p_0),
    \qquad
    \vol(F\setminus \Omega) = \eps.
\end{equation}
Since $\Omega$ is a minimizer for $P_k$ at volume $\vol(\Omega)=\vol(F)$, we deduce
\begin{equation}\label{eq:StimaBoundedness}
\begin{split}
P_k(\Omega)
& \le P_k(F) = P(E) - P(\Omega,M\setminus B_r(p_0)) + \mathcal{H}^{n-1}(\partial B_r(p_0) \cap \Omega) + \int_F f_k\de\vol \\
&\le P(\Omega) + \beta\eps - A(r) + |V'(r)| + \int_F f_k\de\vol.
\end{split}    
\end{equation}
Up to increase $R_3$, we can assume that $\inf_{M\setminus B_r(p_0)} f_k > \sup_{B_{R_1}(p_0)} f_k$. By \eqref{eq:VolumeF} we then get
\[
\begin{split}
    \int_F f_k\de\vol &= \int_{F\cap \Omega} f_k\de\vol + \int_{F\setminus\Omega} f_k\de\vol \le \int_{\Omega \cap B_r(p_0)} f_k\de\vol  + \eps \sup_{B_{R_1}(p_0)} f_k \\& \le \int_{\Omega \cap B_r(p_0)} f_k\de\vol  +\int_{\Omega \setminus B_r(p_0)} f_k\de\vol  = \int_\Omega f_k \de\vol
\end{split}
\]
Combining the latter one with \eqref{eq:StimaBoundedness} we obtain \eqref{eq:GoalBddIsopRegion} with $C=\beta$.
\end{proof}

\subsection{Proof of {\cref{thm:Concavity}}}

We will need the following well known lemma that, roughly speaking, gives an approximation in $H^1$ of the function identically equal to $1$ on the boundary of a set, provided its singular part is sufficiently small. We will apply such a result to a volume constrained minimizer of $P_k$.

\begin{lemma}\label{lem:ApprossimantiVariazione}
Let $\Omega\subset M^n$ be a bounded open set such that its topological boundary $\partial\Omega$ is given by the disjoint union $\partial\Omega = \partial_r\Omega \sqcup \partial_s\Omega$, where $\partial_r\Omega$ is a smooth hypersurface with topological boundary $\partial_s\Omega$ and the Hausdorff dimension of $\partial_s\Omega$ is less than or equal to $n-3$.

Then for any $\delta>0$ there is $\varphi_\delta\in C^\infty_c(\partial_r \Omega)$ such that
\[
 0\le \varphi_\delta \le 1,
 \qquad
    \lim_{\delta\to 0} \int_{\partial\Omega} \varphi_\delta^2\de\mathcal{H}^{n-1} = P(\Omega),
    \qquad
\lim_{\delta\to 0}\int_{\partial \Omega} |\nabla\varphi_\delta|^2\de\mathcal{H}^{n-1} =0,
\]
where $\nabla\varphi_\delta$ here denotes the gradient of $\varphi_\delta$ as a function on the submanifold $\partial_r\Omega$.
\end{lemma}

The above lemma has been proved in the Euclidean setting in \cite[Lemma 2.4]{SternbergZumbrun} under additional regularity assumptions first, and then extended to Riemannian manifolds in \cite[Proposition 2.5, Remark 2.6]{Bayle04}. It is observed in \cite[Lemma 3.1]{MorganRitore02} that the regularity assumptions on $\partial\Omega$ as in the statement are sufficient (see also \cite[Lemma 3.1]{BayleRosales}).

In the following proof of \cref{thm:Concavity}, we are making variations of a volume constrained minimizer $\Omega$ for the penalized perimeter along the normal  field $X = - \varphi \nu_\Omega$, with $\varphi$ given by \cref{lem:ApprossimantiVariazione}. We point out that in dimension $n \leq 7$, by means of the regularity result recalled in  \cref{lem:MinimizingPk}, one could safely consider $X = - \nu_\Omega$.

\begin{proof}[Conclusion of the proof of {\cref{thm:Concavity}}]
For the ease of notation we do not explicitly write the volume forms under the integral sign, since it will be clear from the context. 

Let $\Omega=\Omega(k,\widetilde V)$ be a fixed minimizer of volume $\widetilde V$ for $P_k$. Let $\varphi \in C^\infty_c(\partial_r\Omega)$ be a smooth nonvanishing function with $0\le \varphi\le1$ supported in the regular part of $\partial\Omega$ and define the  normal field $X=-\varphi\nu_\Omega$ and $\phi_t(x)\eqdef \exp(tX(x))$ for $x \in \partial^*\Omega$. Denote by $\Omega_t$ the varied set whose reduced boundary is $\phi_t(\partial^*\Omega)$, for $|t|<\tau$ and $\tau>0$. Define
\[
v(t)\eqdef \vol(\Omega_t),
\qquad
J(v)\eqdef P_k(\Omega_t) \quad \text{for the unique $t$ s.t. $v=v(t)$}.
\]
The definitions of $v(t)$ and $J(v)$ are well posed, indeed
\[
v'(t) = \int_{\partial^*\Omega_t} \scal{-\nu_{\Omega_t}, \partial_t \phi_t},
\]
and then $v'(0)=\int_{\partial\Omega}\varphi>0$ and $v'>0$ for every $t$ if $\tau$ is small enough. Hence $v$ is locally invertible and $J(v)$ is well defined for $v \in (\widetilde V - \widetilde \delta, \widetilde V + \widetilde \delta)$ for some $\widetilde \delta >0$. Let $v^{-1}$ denote the inverse of $v(t)$. We thus have
\begin{equation}\label{eq:Variazioni}
\begin{split}
\frac{\d}{\d t}P(\Omega_t)\bigg|_{t=0} & = \int_{\partial\Omega} \varphi H_\Omega,\\
(v^{-1})'(v) &= \left(\int_{\partial^*\Omega_t} \scal{-\nu_{\Omega_t}, \partial_t \phi_t}\right)^{-1},\\
(v^{-1})'(\widetilde V) &= \left(\int_{\partial\Omega} \varphi\right)^{-1},\\
(v^{-1})''(\widetilde V) &= -\frac{1}{\left( \int_{\partial\Omega} \varphi\right)^2}(v^{-1})'(\widetilde V) v''(0) =-\frac{1}{\left( \int_{\partial\Omega} \varphi\right)^3}\int_{\partial\Omega} \varphi^2 H_\Omega,
\end{split}
\end{equation}
where the last equality follows by computing $v''(0)$ analogously as in \eqref{eq:SecVar1} and \eqref{eq:SecVar2}. Observing that $J(v)=P_k(\Omega_{v^{-1}(v)})$, we can compute
\begin{equation*}
\begin{split}
    \left( \int_{\partial\Omega} \varphi\right)^2 & \frac{\d^2}{\d v^2} J(v) \bigg|_{v=\widetilde V}
    \\&= \left( \int_{\partial\Omega} \varphi\right)^2\frac{\d}{\d v} \left( (v^{-1})' \left( \frac{\d}{\d t}P_k(\Omega_t) \right)\circ v^{-1} \right)\bigg|_{v=\widetilde V} \\
    &=\left( \int_{\partial\Omega} \varphi\right)^2 \left( (v^{-1})''(\widetilde V) \frac{\d}{\d t}P_k(\Omega_t)\bigg|_{t=0} + [(v^{-1})'(\widetilde V)]^2  \frac{\d^2}{\d t^2}P_k(\Omega_t)\bigg|_{t=0}\right).
\end{split}
\end{equation*}
Using \eqref{eq:FirstVariationPotenziale}, \eqref{eq:SecondVariationPotenzialeSpecial}, and \eqref{eq:Variazioni} we get
\begin{equation*}
    \begin{split}
       \left( \int_{\partial\Omega} \varphi\right)^2\frac{\d^2}{\d v^2} J(v) \bigg|_{v=\widetilde V}
    =& -\frac{1}{\int_{\partial\Omega} \varphi}\int_{\partial\Omega} \varphi^2 H_\Omega \left(\int_{\partial\Omega} \varphi H_\Omega + \frac{\d}{\d t}\int_{\Omega_t} f_k \bigg|_{t=0} \right) + \\
    &+\frac{\d^2}{\d t^2}P(\Omega_t)\bigg|_{t=0}  + \frac{\d^2}{\d t^2}\int_{\Omega_t} f_k \bigg|_{t=0}  \\
    =& -\frac{1}{\int_{\partial\Omega} \varphi}\int_{\partial\Omega} \varphi^2 H_\Omega \left(\int_{\partial\Omega} \varphi H_\Omega +\int_{\partial \Omega} \varphi f_k \right) 
    +\\ &+\int_{\partial\Omega} \varphi^2 f_k H_\Omega - \varphi^2 \scal{\nabla f_k , \nu_\Omega}  + \frac{\d^2}{\d t^2}P(\Omega_t)\bigg|_{t=0}.
    \end{split}
\end{equation*}
The second variation formula for the perimeter yields
\[
\frac{\d^2}{\d t^2}P(\Omega_t)\bigg|_{t=0}= \int_{\partial \Omega} |\nabla\varphi|^2 + \varphi^2 \left( H_\Omega^2 - |A_\Omega|^2 - \ric(\nu_\Omega,\nu_\Omega)\right) ,
\]
where $A_\Omega$ is the second fundamental form of $\partial_r \Omega$ and $\nabla\varphi$ is the gradient of $\varphi$ as a function on the submanifold $\partial_r\Omega$. Since $\ric\ge (n-1)K$, $|\nabla f_k|\le \eps_k$, $\varphi\le 1$, and $H_\Omega+ f_k = \lambda \equiv \lambda_{k,\widetilde V} \in \R$ at $\mathcal{H}^{n-1}$ almost every point on $\partial^*\Omega$, we estimate
\begin{equation*}
    \begin{split}
        \left( \int_{\partial\Omega} \varphi\right)^2 & \frac{\d^2}{\d v^2} J(v) \bigg|_{v=\widetilde V}\\
    \le & -\frac{1}{\int_{\partial\Omega} \varphi}\int_{\partial\Omega} \varphi^2 H_\Omega \left(\int_{\partial\Omega}\varphi( H_\Omega + f_k) \right) + \int_{\partial\Omega} \varphi^2 f_k H_\Omega + \eps_k P(\Omega) +\\
    &+\int_{\partial\Omega}\varphi^2 H_\Omega^2 +(n-1)|K|P(\Omega) - \int_{\partial\Omega}  \varphi^2|A_\Omega|^2 + \int_{\partial\Omega} |\nabla\varphi|^2\\
    =&-\lambda\int_{\partial\Omega} \varphi^2H_\Omega + \int_{\partial\Omega} (f_k + H_\Omega)\varphi^2H_\Omega + \eps_k P(\Omega) +\\
    &+(n-1)|K|P(\Omega)- \int_{\partial\Omega}  \varphi^2|A_\Omega|^2 + \int_{\partial\Omega} |\nabla\varphi|^2\\
    =& (\eps_k + (n-1)|K|)P(\Omega)
    - \int_{\partial\Omega}  \varphi^2|A_\Omega|^2 + \int_{\partial\Omega} |\nabla\varphi|^2.
    \end{split}
\end{equation*}
Hence
\begin{equation}\label{eq:J''}
    J''(\widetilde V) \le (\eps_k + (n-1)|K|)\frac{P(\Omega)}{\left( \int_{\partial\Omega} \varphi\right)^2} +\frac{1}{\left( \int_{\partial\Omega} \varphi\right)^2} \left( \int_{\partial\Omega} |\nabla\varphi|^2 -  \int_{\partial\Omega} \varphi^2  |A_\Omega|^2 \right).
\end{equation}
By definition of $I_k$ we have that $I_k(v) \le J(v)$ for any $v \in (\widetilde V - \widetilde \delta, \widetilde V + \widetilde \delta)$ and $I_k(\widetilde V) = J(\widetilde V)$. Then by the latter observation we get
\begin{equation*}
\begin{split}
    \overline{D}^2 I_k (\widetilde V) \eqdef & \limsup_{h\to 0}
    \frac{I_k(\widetilde V +h) + I_k(\widetilde V - h) - 2 I_k(\widetilde V)}{h^2} \le J''(\widetilde V)
\end{split}
\end{equation*}
Now observe that we could take $\varphi=\varphi_\delta$ for any $\delta$ with $\varphi_\delta$ as in \cref{lem:ApprossimantiVariazione}. Hence letting $\delta\to0$, using that $P(\Omega) \ge I(\widetilde V)$, and exploiting \eqref{eq:J''}, we conclude that
\begin{equation}\label{eq:D2IkVtilde}
\overline{D}^2 I_k (\widetilde V) \le  \frac{\eps_k + (n-1)|K|}{P(\Omega)} 
    \le   \frac{\eps_k + (n-1)|K|}{I(\widetilde V)}. 
\end{equation}

Recalling that $I$ is continuous and strictly positive (\cite[Corollary 2]{FloresNardulli20} and \cite[Remark 4.7]{AFP21}), up to taking a smaller $\widetilde \delta$, we have that $I(V)$ is bounded below by a strictly positive constant independent of $V \in (\widetilde V - \widetilde \delta, \widetilde V + \widetilde \delta)$.
Since $\widetilde V$ above is arbitrary, we deduce that $\overline{D}^2 I_k (V) \le 2C_{\widetilde V}<+\infty$ for every $V \in (\widetilde V - \widetilde \delta, \widetilde V + \widetilde \delta)$.
The latter, together with the fact that $I_k$ is continuous, see \cref{lem:ApprossimazioneProfili}, implies that $I_k-C_{\widetilde V}V^2$ is concave on $(\widetilde V - \widetilde \delta, \widetilde V + \widetilde \delta)$.

Passing to the limit $k\to+\infty$, as $I_k\to_k I$ locally uniformly by \cref{lem:ApprossimazioneProfili}, we deduce that $I-C_{\widetilde V}V^2$ is concave on $(\widetilde V - \widetilde \delta, \widetilde V + \widetilde \delta)$ as well. Also, if $K=0$, \eqref{eq:D2IkVtilde} implies that $I$ is concave. In any case, the profile $I$ is twice differentiable almost everywhere, and if $\widetilde V$ is such a point, \eqref{eq:D2IkVtilde} implies that
\[
I''(\widetilde V) \le \frac{(n-1)|K|}{I(\widetilde V)}.
\]
\end{proof}

It would be interesting to understand how the above proof could be improved in order to get concavity properties on the function $I_{(M^n,g)}^{\frac{n}{n-1}}$. Under the a priori assumption of existence of isoperimetric sets, it has been studied in  \cite{BavardPansu86, Bayle03, Bayle04, BayleRosales, MondinoNardulli16}.

\subsection{Consequences of the concavity properties of the isoperimetric profile}
In this section we collect some consequences of \cref{thm:Concavity}. First, since concave functions are locally Lipschitz, we deduce the following useful corollary, which improves the main result in \cite[Theorem 2]{FloresNardulli20}. 

\begin{cor}
\label{cor:loclip}
Suppose that $(M^n,g)$ is noncollapsed and that $\ric \ge (n-1)K$ for some $K \in \R$. Then the isoperimetric profile $I_{(M^n,g)}$ is a locally Lipschitz function.
\end{cor}

In the following result we exploit the concavity of $I$ under nonnegative Ricci curvature, together with the sharp isoperimetric inequality \eqref{eq:isoperimetricCD}, to deduce the precise asymptotic expansion of $I$ and $I'$ at infinity when $\ric\geq 0$.

\begin{cor}\label{lem:IsopAVR}
Let $(M^n,g)$ be a noncompact Riemannian manifold with $\ric\geq 0$, and denote by $I$ its isoperimetric profile. Then 
\begin{equation}\label{eq:AsintoticaI}
\lim_{v\to+\infty}\frac{I(v)}{v^{(n-1)/n}}=n(\omega_n\mathrm{AVR}(M^n,g))^{1/n}.    
\end{equation}
Moreover the right derivative $I'_+(v)\eqdef \lim_{h\to0^+} \frac{I(v+h)-I(v)}{h}$ exists at any $v>0$ and it satisfies
\begin{equation}\label{eq:AsintoticaI'}
\lim_{v\to+\infty}v^{1/n}I'_+(v)=(n-1)(\omega_n\mathrm{AVR}(M^n,g))^{1/n}.
\end{equation}
\end{cor}

\begin{proof}
For simplicity, throughout the proof we denote $\theta:=\mathrm{AVR}(M^n,g)$. Let us first prove \eqref{eq:AsintoticaI}. We first deal with the case $\theta>0$. Let us fix $p\in M^n$. From Bishop--Gromov comparison theorem, see \cref{thm:BishopGromov}, we know that $P(B_r(p))/(n\omega_nr^{n-1})$ is almost everywhere non-increasing on $[0,+\infty)$. Since $\vol(B_r(p))/(\omega_nr^n)\to\theta$, as $r\to+\infty$, a standard use of the coarea formula for $\vol(B_r(p))$, together with the monotonicity ensured by the Bishop--Gromov comparison result, imply that the essential limit of $P(B_r(p))/(n\omega_nr^{n-1})$ as $r\to+\infty$ is $\theta$ as well. Hence, for every $\varepsilon>0$ there exists $r_0:=r_0(\varepsilon)$ such that for almost every $r\geq r_0$ we have 
$$
P(B_r(p))\leq(1+\varepsilon)\theta n\omega_n r^{n-1}.
$$
Moreover, by approximating balls from the outside and by using the semicontinuity of the perimeter, the previous inequality holds for {\em all} $r\geq r_0$. Hence for every $V$ sufficiently large, there exists a ball of center $p$, radius $r:=r(V)\geq r_0$ and volume $V$ such that
$$
I(V)\leq P(B_r(p))\leq (1+\varepsilon)\theta n\omega_nr^{n-1}\leq (1+\varepsilon)n(\omega_n\theta)^{1/n}V^{(n-1)/n},
$$
where in the last inequality we used $\theta\omega_nr^n\leq V$, which comes from Bishop--Gromov comparison. Since $I(V)\geq n(\omega_n\theta)^{1/n}V^{(n-1)/n}$ for every $V>0$, as an immediate outcome of \cref{thm:IsoperimetriSharpBrendle}, we get \eqref{eq:AsintoticaI} in the case $\theta>0$.

Let us now prove \eqref{eq:AsintoticaI} when $\theta=0$. For every $V>0$ there exists a radius $r:=r(V)$ such that $\vol(B_{r(V)}(p))=V$. Notice that $r(V)\to +\infty$ when $V\to +\infty$. Moreover 
\begin{equation}\label{eqn:EstimateOnIV}
\begin{split}
\frac{I(V)}{V^{(n-1)/n}}&\leq \frac{P(B_{r}(p))}{\vol(B_{r}(p))^{(n-1)/n}}=\frac{P(B_r(p))}{\vol(B_r(p))}\left(\frac{\vol(B_r(p))}{\omega_n r^n}\right)^{1/n}r\omega_n^{1/n}\\
&\leq n\omega_n^{1/n}\left(\frac{\vol(B_r(p))}{\omega_n r^n}\right)^{1/n},
\end{split}
\end{equation}
where in the first inequality we used the definition of isoperimetric profile, and in the last inequality we used the comparison in \eqref{3}. Since $\theta=0$ we have that for every $\varepsilon$ there exists $r_0:=r_0(\varepsilon)$ such that $\vol(B_r(p))/(\omega_nr^n)\leq \varepsilon$ for every $r\geq r_0$. Hence from \eqref{eqn:EstimateOnIV} we conclude the sought limit in \eqref{eq:AsintoticaI} also in the case $\theta=0$.

Now we can prove \eqref{eq:AsintoticaI'}. If $M$ is not noncollapsed, then \eqref{eq:AsintoticaI'} is trivially verified due to the equivalence in \cref{lem:NonSoDoveMetterlo}. So let us prove \eqref{eq:AsintoticaI'} when $M^n$ is noncollapsed. By \cref{thm:Concavity} we know that $I$ is concave, which implies the existence of $I'_+$ everywhere and that $I'_+$ is nonincreasing. 

We claim that, for any $\eps>0$ and $0<\delta<2$ there exists $\bar v= \bar v(\eps, \delta)$ such that 
\begin{equation}\label{zz2}
    \left| \frac{I(v(1+\delta)) - I(v)}{(v(1+\delta))^{\frac{n-1}{n}} - v^{\frac{n-1}{n}} }
    - n(\omega_n \theta)^{1/n}\right|
    \le \eps 
    \quad \text{for any $v\ge \bar v$.}
\end{equation}
Indeed the latter follows from \eqref{eq:AsintoticaI} and the inequality
\begin{align*}
    & \left| \frac{I(v(1+\delta)) - I(v)}{(v(1+\delta))^{\frac{n-1}{n}} - v^{\frac{n-1}{n}} }
    - n(\omega_n \theta)^{1/n}\right| 
    \\& \le 
     \frac{(1+\delta)^{\frac{n-1}{n}}}{(1+\delta)^{\frac{n-1}{n}}-1}\left|\frac{I(v(1+\delta))}{(v(1+\delta))^{\frac{n-1}{n}}} - n(\omega_n \theta)^{1/n}\right|
    +
     \frac{1}{(1+\delta)^{\frac{n-1}{n}}-1}\left|\frac{I(v)}{v^{\frac{n-1}{n}}} - n(\omega_n \theta)^{1/n}\right| \, .
\end{align*}
The monotonicity of $I_+'$ gives
\begin{equation}
    I(v(1+\delta)) - I(v) = \int_{v}^{v(1+\delta)} I'_+(t) dt \ge \delta v I_+'(v(1+\delta)) \, ,
\end{equation}
and
\begin{equation}
    I(v(1+\delta)) - I(v) = \int_{v}^{v(1+\delta)} I'_+(t) dt \le \delta v I_+'(v) \, ,
\end{equation}
for all $v>0$, where we also employed the fact that $I$ is Lipschitz, as observed in \cref{cor:loclip}, in order to apply the Fundamental Theorem of Calculus.
Hence, we get
\begin{equation}\label{z2}
\begin{split}
        \frac{I(v(1+\delta)) - I(v)}{(v(1+\delta))^{\frac{n-1}{n}} - v^{\frac{n}{n-1}}}
    &\ge 
    \frac{\delta v\, I'_+(v(1+\delta))}{\frac{n-1}{n}\int_v^{v(1+\delta)} t^{-1/n} dt}\\
    &\ge
    (1+\delta)^{-1/n}\frac{n}{n-1}  (v(1+\delta))^{1/n} I'_+(v(1+\delta)) \, ,
\end{split}
\end{equation}
and
\begin{equation}\label{z3}
    \frac{I(v(1+\delta)) - I(v)}{(v(1+\delta))^{\frac{n-1}{n}} - v^{\frac{n}{n-1}}}
    \le 
    \frac{\delta v\, I'_+(v)}{\frac{n-1}{n}\int_v^{v(1+\delta)} t^{-1/n} dt}
    \le
    (1+\delta)^{1/n} \frac{n}{n-1}  v^{1/n} I'_+(v) \, ,
\end{equation}
for all $v>0$. The latter, together with \eqref{zz2}, implies
\begin{align*}
    (1+\delta)^{-1/n} (n-1)&(\omega_n\theta)^{1/n} - C(n)\eps 
    \\& \le  \liminf_{v\to \infty} v^{1/n}I_+'(v) 
     \\& \le
    \limsup_{v\to \infty} v^{1/n}I_+'(v)
    \\& \le
     (1+\delta)^{1/n} (n-1)(\omega_n\theta)^{1/n} + C(n)\eps \, .
\end{align*}
Letting $\eps, \delta\to 0$ we get the sought conclusion.
\end{proof}

\begin{remark}
We remark that \eqref{eq:AsintoticaI} holds for arbitrary $\CD(0,n)$ spaces since it only relies on Bishop--Gromov monotonicity, see \cref{rem:PerimeterMMS2}, and \cref{thm:IsoperimetriSharpBrendle}.
\end{remark}

In the following corollary we observe that for an arbitrary manifold with nonnegative Ricci curvature and Euclidean volume growth the isoperimetric profile is strictly increasing. This should be compared with \cite[Theorem 3.3]{Rit17} where the author shows that on every complete noncompact Riemannian manifold with strictly positive sectional curvature the isoperimetric profile is strictly increasing.  

\begin{cor}\label{cor:ProfileIncreasing}
Let $(M^n,g)$ be a noncompact Riemannian manifold with $\ric\geq 0$ and $\mathrm{AVR}(M^n,g)>0$. Then the isoperimetric profile $I_{(M^n,g)}$ is strictly increasing.
\end{cor}

\begin{proof}
Since $I_{(M^n,g)}$ is concave due to \cref{thm:Concavity}, we have that $I'_+$ is nonincreasing. Moreover, since we have the asymptotic relation in \eqref{eq:AsintoticaI} and $\mathrm{AVR}(M^n,g)>0$, we conclude that $I'_+\geq 0$. Hence, if for some $v_0>0$ we have $I'_+(v_0)=0$, therefore $I'_+(v)=0$ for every $v\geq v_0$, which is in contradiction with \eqref{eq:AsintoticaI}. In conclusion $I'_+>0$ everywhere, and hence $I$ is strictly increasing.
\end{proof}

\section{Main isoperimetric existence results}\label{sec:SectionExistence}

In this section we prove the main existence result \cref{thm:Principale} and the consequent \cref{cor:Princ2} and \cref{thm1}.

\subsection{Proof of {\cref{thm:Principale}}}
Let us denote for simplicity $\theta:=\mathrm{AVR}(M^n,g)$. We divide the proof in two steps. The first step is the following Lemma in which we prove that, under the hyptoheses of \cref{thm:Principale}, for large volumes $V$ any minimizing sequence of bounded finite perimeter sets of volume $V$ can lose at most one piece at infinity along $N$.

\begin{lemma}\label{lem:Step1}
Let $k\geq 0$ be a natural number, and let $(M^n,g)=(\mathbb R^k\times N^{n-k},g_{\mathbb R^k}+ g_N)$ be a complete noncollapsed Riemannian manifold with $\ric\geq 0$. Assume that $M$ satisfies the hypothesis of \cref{thm:Principale}. Then there exists $\overline{V}>0$ such that the following holds. 

For every $V\geq \overline{V}$, let $\{\Omega_i\}_{i\in\mathbb N}$ be a minimizing sequence for the volume $V$ of bounded finite perimeter sets. Let $p_{i,j}:=(t_{i,j},x_{i,j})\in\mathbb R^k\times N$ be the points for which \cref{thm:MassDecomposition} applied to $\{\Omega_i\}_{i\in\mathbb N}$ holds, using the same notation therein, and for a fixed $x_0\in N$ let 
\begin{equation}\label{eqn:J1J2}
\begin{split}
    J_1\eqdef &\{j:\sup_i \dist_N(x_0,x_{i,j})<+\infty\}, \\
    J_2\eqdef &\{j:\limsup_i \dist_N(x_0,x_{i,j})= +\infty\}.
\end{split}
\end{equation}
Then for every $j\in J_2$ we have $\lim_i \vol(\Omega_{i,j}^d) > V/2$. In particular, $J_2$ consists of at most one element for any $V \geq \overline{V}$.
\end{lemma}

\begin{proof}
Let us observe first that \eqref{eq:AsintoticaI'} implies that there exists $\widehat{V}>0$ such that
\begin{equation*}
    \begin{split}
        I'_{+}(v) < \frac{n}{n-1} (n-1)(\omega_n(\theta+ \varepsilon/2))^{1/n} v^{-1/n}& \qquad \forall v \ge \widehat{V}, \\
    \end{split}
\end{equation*}
where the constant $ \varepsilon$ is the one provided by the hypothesis in \cref{thm:Principale} .
In particular
\begin{equation*}
    \begin{split}
        I(V)-I(V-\delta) &< \int_{V-\delta}^V \frac{n}{n-1} (n-1)(\omega_n(\theta+ \varepsilon/2))^{1/n} v^{-1/n} \de v \\
       & \leq n(\omega_n(\theta+ \varepsilon/2))^{1/n} \frac{\delta}{(V-\delta)^{1/n}},
    \end{split}
\end{equation*}
whenever $V>V-\delta \ge \widehat{V}$. Also, we have that $\delta/(V-\delta)^{1/n} \le \delta^{1-1/n}$ if $\delta\le V/2$. Therefore
\begin{equation}\label{eq:StimaConcav}
    I(V)-I(V-\delta) < n(\omega_n(\theta+ \varepsilon/2))^{1/n}\delta^{1-1/n},
\end{equation}
whenever $V>V-\delta \ge \widehat{V}$ and $\delta\le V/2$.

So we choose $\overline{V}\eqdef 2\widehat{V}$. Let $V\ge \overline{V}$ and $\{\Omega_i\}_{i\in\mathbb N}$ be an arbitrary minimizing sequence of bounded finite perimeter sets of volume $V$. Let $\Omega_{i,j}^d$ be defined as in \cref{thm:MassDecomposition}, and we adopt the notation therein. Suppose by contradiction that $\meas_j(Z_j)=\lim_i \vol( \Omega_{i,j}^d) \le V/2$ for some $j\in J_2$. From \cref{thm:MassDecomposition} we have
\[
    I(V) = P_{X_j}(Z_j) + P(\Omega) + \sum_{\ell\neq j} P_{X_\ell}(Z_\ell) = P_{X_j}(Z_j)+ I(\vol(\Omega)) + \sum_{\ell\neq j} I_{X_\ell}(\meas_\ell(Z_\ell)) .
\]
By \cref{thm:IsoperimetriSharpBrendle}, by hypothesis, and by \cite[Proposition 3.2]{AFP21} we then get
\[
\begin{split}
    I(V) &\ge n(\omega_n(\theta+ \varepsilon))^{1/n}\meas_j(Z_j)^{(n-1)/n} + I(\vol(\Omega)) + \sum_{\ell\neq j} I_{X_\ell}(\meas_\ell(Z_\ell))
\\
&\ge n(\omega_n(\theta+ \varepsilon))^{1/n}\meas_j(Z_j)^{(n-1)/n}  + I \left(\vol(\Omega) + \sum_{\ell\neq j} \meas_\ell(Z_\ell) \right)
\\
&= n(\omega_n(\theta+ \varepsilon))^{1/n}\meas_j(Z_j)^{(n-1)/n}  + I\left (V -  \meas_j(Z_j) \right).
\end{split}
\]
By the absurd hypothesis we can take $\delta= \meas_j(Z_j)$ in \eqref{eq:StimaConcav} contradicting the above estimate, and thus completing the proof of the claim.
\end{proof}

Let us now conclude the proof of \cref{thm:Principale} by exploiting the previous Lemma. Let us fix $p:=(0,x_0)\in\mathbb R^k\times N$. We now show that there exists $V_0\geq \overline V$ such that for every $V\geq V_0$ there exists an isoperimetric region of volume $V$, thus concluding the proof of the Theorem.

To reach the latter conclusion we start proving that there exists $V_0\geq \overline V$ such that the following improvement of \cref{lem:Step1} holds. 
Let $V\geq V_0$, and let $\{\Omega_i\}_{i\in\mathbb N}$ be a minimizing sequence for the volume $V$ of bounded finite perimeter sets. Let $p_{i,j}:=(t_{i,j},x_{i,j})\in\mathbb R^k\times N$ be the points for which \cref{thm:MassDecomposition} applied to $\{\Omega_i\}_{i\in\mathbb N}$ holds, using the same notation therein, and define 
\[
\begin{split}
    J_1\eqdef &\{j:\sup_i \dist_N(x_0,x_{i,j})<+\infty\}, \\
    J_2\eqdef &\{j:\limsup_i \dist_N(x_0,x_{i,j})= +\infty\},
\end{split}
\]
We claim that $J_2$ is actually empty. 

Let us prove the previous claim by contradiction, after having provided the precise value of $V_0$. Let $ \varepsilon$ be the constant given by the hypothesis of \cref{thm:Principale}. Let us first consider the case $\theta>0$. Let us define 
$$
\varepsilon^*:=(1+ \varepsilon/\theta)^{1/n}-1.
$$
Let $V_0\geq \overline V$ be big enough, where $\overline V$ is provided by \cref{lem:Step1}, such that
\begin{equation}\label{eqn:ContrNEW}
I(V)\leq n(\omega_n\theta)^{1/n}V^{(n-1)/n}\left(1+2^{-1}\varepsilon^*(1/3)^{(n-1)/n}\right), \qquad \text{for all $V\geq V_0$.}
\end{equation}
Such a value of $V_0$ exists thanks to \cref{lem:IsopAVR}. From now on, we fix $V\geq V_0$ and we fix a minimizing sequence $\{\Omega_i\}_{i\in\mathbb N}$ of bounded finite perimeter sets of volume $V$ for which, by contradiction, there exists exactly one $j_0\in J_2$, by \cref{lem:Step1}. From \eqref{eq:isoperimetricCD} and the hypothesis we have
\begin{equation}\label{eqn:EstimateNEW}
\frac{P_{X_{j_0}}(Z_{j_0})}{\meas_{j_0}(Z_{j_0})^{(n-1)/n}}\geq n\omega_n^{1/n}\mathrm{AVR}(X_{j_0},\dist_{j_0},\meas_{j_0})^{1/n}\geq n\omega_n^{1/n}(\theta+ \varepsilon)^{1/n}.
\end{equation}

From \cref{lem:Step1} and from item (iii) of \cref{thm:MassDecomposition} we have
\begin{equation}\label{eqn:MobbastaveramenteperoNEW}
\meas_{j_0}(Z_{j_0})=\lim_i\vol(\Omega_{i,{j_0}}^d)\geq V/3.
\end{equation}
For every $j\in J_1$, we have that $(M,\dist,p_{i,j},\vol)\to_{i}(M,\dist,\overline p_j,\vol)$ in the pmGH topology, for some $\overline p_j\in M$. Indeed, since $j\in J_1$ we have that, up to subsequences, $x_{i,j}\to_{i} \overline x_j\in N$. Hence it suffices to choose $\overline p_j:=(0,\overline x_j)$ thanks to the homogeneity in the factor $\mathbb R^k$. Hence, by the result in \cref{thm:MassDecomposition}, we have that $\Omega_{i,j}\to_i Z_j$, where $Z_j$ is an isoperimetric region in $M$. Now,
by exploiting item (iv) of \cref{thm:MassDecomposition}, taking into account \eqref{eqn:EstimateNEW} and \eqref{eqn:MobbastaveramenteperoNEW}, exploiting also the isoperimetric inequality in \cref{thm:IsoperimetriSharpBrendle}, we finally get, recalling that $\varepsilon^*:=(1+ \varepsilon/\theta)^{1/n}-1$, that
\begin{equation}\label{eq:MainEstimateJ2}
    \begin{split}
        I(V)&=P(\Omega)+\sum_{j\in J_1}P(Z_j) + P_{X_{j_0}}(Z_{j_0}) \\ &\geq n(\omega_n\theta)^{1/n}\left(\vol(\Omega)^{(n-1)/n}+\sum_{j\in J_1}\vol(Z_j)^{(n-1)/n}+(1+ \varepsilon/\theta)^{1/n}\meas_{j_0}(Z_{j_0})^{(n-1)/n}\right) \\
        &\geq n(\omega_n\theta)^{1/n}\left(\left(\vol(\Omega)+\sum_{j=1}^{\overline N}\meas_j(Z_j)\right)^{(n-1)/n}+\varepsilon^*\meas_{j_0}(Z_{j_0})^{(n-1)/n}\right) \\
        &\geq n(\omega_n\theta)^{1/n}V^{(n-1)/n}\left(1+\varepsilon^*(1/3)^{(n-1)/n}\right),
    \end{split}
\end{equation}
which is a contradiction with \eqref{eqn:ContrNEW}. Hence $J_2$ is empty. 

Let us now consider the case $\theta=0$. From \eqref{eq:AsintoticaI} we have that there exists $V_0\geq \overline V$ such that for all $V\geq V_0$ we have 
\begin{equation}\label{eqn:BASTABASTA}
I(V)\leq n(\omega_n2^{-1}(1/3)^{(n-1)/n} \varepsilon)^{1/n}V^{(n-1)/n}.
\end{equation}
Arguing similarly as before we have, for every $V\geq V_0$, and for every minimizing sequence $\{\Omega_i\}_{i\in\mathbb N}$ of bounded sets of finite perimeter of volume $V$, by using the same notation as before,
\begin{equation}\label{eq:ww}
    \begin{split}
        I(V)&=P(\Omega)+\sum_{j\in J_1}P(Z_j)+P_{X_{j_0}}(Z_{j_0}) \geq P_{X_{j_0}}(Z_{j_0}) \\
        &\geq n(\omega_n \varepsilon)^{1/n}\meas_{j_0}(Z_{j_0})^{(n-1)/n}\geq  n(\omega_n(1/3)^{(n-1)/n} \varepsilon)^{1/n}V^{(n-1)/n},
    \end{split}
\end{equation}
which is a contradiction with \eqref{eqn:BASTABASTA}, and thus $J_2$ is empty also in this case.

Let us now take $V_0\geq \overline V$ such that the previous claim about $J_2$ to be empty (for every minimizing sequence of bounded sets of finite perimeter of volume greater than $V_0$) holds. Let us now conclude that for every $V\geq V_0$ there exists an isoperimetric region of volume $V$ in $M$. Indeed, let us take a minimizing sequence $\{\Omega_i\}_{i\in\mathbb N}$ of bounded finite perimeter sets of volume $V$. We have already shown that for every $j\in J_1$, and thus for all $j$'s since $J_2$ is empty as $V\geq V_0$, we have that $(M,\dist,p_{i,j},\vol)\to_{i}(M,\dist,\overline p_j,\vol)$ in the pmGH topology, for some $\overline p_j\in M$. From \cref{cor:IsopBounded} we get that $Z_j$ is bounded for every $j$ since it is an isoperimetric region in $M$. Hence, by properly translating $Z_j$ along the coordinate $\mathbb R^k$ into some $Z_j'$, we can define $\Omega':=\sqcup_{j=1}^{\overline N} Z'_j$, where the $Z_j'$s are mutually disjoint. Moreover, from \eqref{eq:UguaglianzeIntro} we get that $\vol(\Omega')=V$. Hence, from \eqref{eq:UguaglianzeIntro}, we conclude that 
\begin{equation}
\label{condizione-critica}
I(V)=P(\Omega)+\sum_{j=1}^{\overline N} P(Z_j)=P(\Omega')\geq I(V),
\end{equation}
from which $\Omega'$ is the sought isoperimetric region of volume $V$.

\subsection{Proof of {\cref{cor:Princ2}}}

In order to complete the proof of \cref{cor:Princ2} we will use \cref{thm:Principale} together with the following Lemma, that we state and prove in the general $\RCD$ setting. In item (i) of \cref{lem:DensitaStaccata}
we show that when the hypothesis of \cref{cor:Princ2} holds, then the density of the points at distance one from the tips of the asymptotic cones of $M^n$ is uniformly greater than $\mathrm{AVR}(M^n,g)$. In item (ii) of \cref{lem:DensitaStaccata} we show that when the previous conclusion on the density holds, then every pmGH limit at infinity has $\mathrm{AVR}$ uniformly greater than the $\mathrm{AVR}(M^n,g)$.

\begin{lemma}\label{lem:DensitaStaccata}
Let $k\geq 0$ be a natural number, let $(Y,\dist_Y,\meas_Y)$ be a metric measure space,  and assume
$$
(X:=\mathbb R^k\times Y,\dist:=\dist_{\mathbb R^k}\otimes\dist_Y,\meas:=\meas_{\mathbb R^k}\otimes \meas_Y),
$$
is an $\RCD(0,n)$ space. Assume 
\begin{equation}\label{eqn:HypothesisTheta}
\mathrm{AVR}(X,\dist,\meas)>0.
\end{equation}
Then the following two statements hold.
\begin{itemize}
\item[(i)] If there exists no asymptotic cone of $(Y,\dist_Y,\meas_Y)$ that splits a line, then there exists an ${\varepsilon}>0$ such that for every asymptotic cone 
$$
(\mathbb R^k\times C,\dist_{\mathbb R^k}\otimes\dist_C,\meas_{\mathbb R^k}\otimes\meas_C,(0,v_\infty)),
$$ 
of $(X,\dist,\meas)$, and for every $p\in \mathbb R^k\times C$ with $\dist_{\mathbb R^k}\otimes\dist_C(p,\mathbb R^k\times\{v_\infty\})=1$, we have
$$
\,\,\,\lim_{r\to 0}\frac{\meas_{\mathbb R^k}\otimes\meas_ C(B_r(p))}{\omega_nr^n}=:\vartheta[(\mathbb R^k\times C,\dist_{\mathbb R^k}\otimes\dist_C,\meas_{\mathbb R^k}\otimes\meas_C,p)]\geq\mathrm{AVR}(X,\dist,\meas)+{\varepsilon}.
$$
\item[(ii)] 
Assume that there exists $\alpha > 0$ such that
$$
\lim_{r\to 0}\frac{\meas_{\mathbb R^k}\otimes\meas_ C(B_r(p))}{\omega_nr^n}\geq\alpha,
$$
for every asymptotic cone 
$$
(\mathbb R^k\times C,\dist_{\mathbb R^k}\otimes\dist_C,\meas_{\mathbb R^k}\otimes\meas_C,(0,v_\infty)),
$$ 
of $(X,\dist,\meas)$ and for every $p\in \mathbb R^k\times C$ with $\dist_{\mathbb R^k}\otimes\dist_C(p,\mathbb
R^k\times\{v_\infty\})=1$.
Fix $x_0:=(0,y_0) \in \mathbb R^k\times Y$. Fix $\{x_i:=(t_i,y_i)\}_{i\geq 1}\subset X$ a sequence with $\dist_Y(y_i,y_0)\to_i +\infty$, and let $X_{\infty}$ be a pmGH limit of a subsequence of $(X,\dist,\meas,x_i)$, namely 
$$
(X,\dist,\meas,x_i)\to_i(X_{\infty},\dist_\infty,\meas_\infty,x_\infty),
$$
in the pmGH sense up to a subsequence.
Hence, $\mathrm{AVR}(X_{\infty},\dist_\infty,\meas_\infty)\geq\alpha$.
\end{itemize}
\end{lemma}

\begin{proof}
In the proof for simplicity we denote $\theta_X:=\mathrm{AVR}(X,\dist,\meas)$. Inductively using \cite[Theorem 7.4]{Gigli13}, one deduces that $(Y,\dist_Y,\meas_Y)$ is an $\RCD(0,n-k)$ space. Moreover, one can check that ${\rm AVR}(Y,\dist_Y,\meas_Y)>0$ as well.

\vspace{0.2cm}
Let us first prove item (i). Let us assume by contradiction that there exist, for every $i\in\mathbb N$, asymptotic cones $(C_i',\dist_i,\meas_i,(0,v_{\infty,i}))$ of $(X,\dist,\meas)$, and points $p_i\in C_i'$ such that $\dist_i(p_i,\mathbb R^k\times\{v_{\infty,i}\})=1$ and 
$$
\vartheta[(C_i',\dist_i,\meas_i,p_i)]\leq \theta_X+1/i.
$$
Recall that $C_i'=\mathbb R^k\times C_i$, for some cone $C_i$ with tip $v_{\infty,i}$. Notice that by translating along $\mathbb R^k$ we may assume that the $\mathbb R^k$ coordinate of $p_i$ is $0$.
Let us first notice that the sequence $\{(C_i',\dist_i,\meas_i,(0,v_{\infty,i}))\}_{i\in\mathbb N}$ is precompact in the pmGH topology. Indeed, since any $(C_i',\dist_i,\meas_i,(0,v_{\infty,i}))$ is the pmGH limit of a rescaling of $X$, the $C_i'$'s are $\RCD(0,n)$ spaces with $\meas_i(B_1(0,v_{\infty,i}))=\theta_X$ for every $i\in\mathbb N$, due to the volume convergence result in \cref{thm:volumeconvergence}. Hence precompactness follows from \cref{rem:GromovPrecompactness}. Moreover, any limit of such a sequence of asymptotic cones must be an asymptotic cone of $(X,\dist,\meas)$, due to the fact that the pmGH-topology, in our case, is metrizable, see \cite[Theorem 3.15]{GigliMondinoSavare15}. Indeed, by exploiting the latter information, a diagonal argument implies that any pmGH limit point of the sequence $\{(C_i',\dist_i,\meas_i,(0,v_{\infty,i}))\}_{i \in \N}$ is also an asymptotic cone of $X$.

Hence, let us take $(C',\dist_{C'},\meas_{C'},(0,v_{\infty}))$ an asymptotic cone of $(X,\dist,\meas)$, that is also a pmGH limit of the sequence $\{(C_i',\dist_i,\meas_i,(0,v_{\infty,i}))\}_{i\in\mathbb N}$. In a proper realization of such a pmGH limit, we can assume that, up to further subsequences, $p_i\to p\in C'$, and since $\dist_i(p_i,\mathbb R^k\times\{v_{\infty,i}\})=1$, we get by continuity of the distance that $\dist(p,\mathbb R^{k}\times\{v_{\infty}\})=1$, and hence $p\notin \mathbb R^k\times\{ v_{\infty}\}$. Since $(0,v_{\infty})$ is one tip of the asymptotic cone $C'$, from the volume convergence result in \cref{thm:volumeconvergence}, we conclude that 
\begin{equation}\label{eqn:QuellaCheConclude}
\meas_{C'}(B_r(0,v_{\infty}))=\theta_X\omega_nr^n,
\end{equation}
for every $r>0$. By an immediate consequence of Bishop--Gromov volume comparison and the volume convergence we also have that 
\begin{equation}\label{eqn:Maggiore}
\vartheta[(C',\dist_{C'},\meas_{C'},p)]\geq \theta_X.
\end{equation}
Moreover, since the density $\vartheta$ is lower semicontinuous with respect to the pmGH convergence, see \cite[Lemma 2.2 (i)]{DePhilippisGigli18}, we also conclude that 
\begin{equation}\label{eqn:Minore}
\vartheta[(C',\dist_{C'},\meas_{C'},p)]\leq\liminf_{i\to+\infty}\vartheta[(C_i',\dist_i,\meas_i,p_i)]\leq \theta_X.
\end{equation}
Hence, from \eqref{eqn:Maggiore} and \eqref{eqn:Minore}, we deduce that 
\begin{equation}\label{eqn:ThetaA0}
\lim_{r\to 0}\frac{\meas_{C'}(B_r(p))}{\omega_nr^n}=:\vartheta[(C',\dist_{C'},\meas_{C'},p)]=\theta_X.
\end{equation}
Since $p$ and $(0,v_{\infty})$ are at finite distance, from \eqref{eqn:QuellaCheConclude}, a simple argument involving the triangle inequality, and the monotonicity of Bishop--Gromov ratios (see \cref{rem:PerimeterMMS2}) we also deduce that 
\begin{equation}\label{eqn:ThetaAInfinito}
\lim_{r\to +\infty}\frac{\meas_{C'}(B_r(p))}{\omega_nr^n}=\theta_X.
\end{equation}
From \eqref{eqn:ThetaA0}, \eqref{eqn:ThetaAInfinito}, and the monotonicity of the Bishop--Gromov ratios, see \cref{rem:PerimeterMMS2}, we deduce that $\meas_{C'}(B_r(p))=\theta_X\omega_nr^n$ for every $r>0$. Hence, from the result in \cite[Theorem 1.1]{DePhilippisGigli16}, we conclude that $C'$ is a metric cone with tip $p\notin \mathbb R^k\times\{ v_{\infty}\}$. Hence $C'=\mathbb R^k\times C$ is also metric cone over a tip that is not in $\mathbb R^k\times\{v_\infty\}$, and then from \cite[Proposition 1.18]{AntBruSem19} we get that $C$ splits a line. But this is not possible, since $C$ is an asymptotic cone of $Y$ that by hypothesis does not split a line, thus giving the sought contradiction.

\vspace{0.2cm}
Let us now prove item (ii). Since $(X,\dist,\meas)$ is an $\RCD(0,n)$ space we conclude by stability that also $(X_\infty,\dist_\infty,\meas_\infty)$ is an $\RCD(0,n)$ space. Let us denote $\widetilde x_{i}:=(0,y_i)$ and observe that $(X,\dist,\meas,x_i)$ is isomorphic to $(X,\dist,\meas,\widetilde x_i)$. 
Then we may suppose that the pmGH limit $(X_\infty,\dist_\infty,\meas_\infty,x_\infty)$ is obtained through (a sub)sequence of $(X,\dist,\meas,\widetilde x_i)$.

Since we have that $\dist_Y(y_i,y_0)\to_i +\infty$, if we set $\rho_i:=\dist(x_0,\widetilde x_i)$, we get that, up to subsequences in $i$, $(X,\rho_i^{-1}\dist,\rho_i^{-n}\meas,x_0)$ pmGH-converges to an asymptotic cone \\ $(C',\dist_{C'},\meas_{C'},(0,v_{\infty}))$. Notice that $C'=\mathbb R^k\times C$ for some cone $C$ with tip $v_\infty$. Moreover, in a realization of such a convergence, $\widetilde x_i\to_{i} p$ such that $\dist_{C'}(p,\mathbb R^k\times\{z_\infty\})=1$ from the fact that both $x_0$ and $\widetilde x_i$ have zero component along $\mathbb R^k$. From the hypothesis we get that, for some $\alpha>0$, we have
$$
\lim_{r\to 0}\frac{\meas_{C'}(B_r(p))}{\omega_nr^n}\geq\alpha.
$$
Hence, for every $\eta$ there exists some $\delta:=\delta(\eta)>0$ for which
\begin{equation}\label{eqn:LEI1New}
\frac{\meas_{C'}(B_\delta(p))}{\omega_n\delta^n} \geq \alpha-\eta.
\end{equation}
Since $(X,\rho_i^{-1}\dist,\rho_i^{-n}\meas,\widetilde x_i)$ pmGH-converges to $(C',\dist_{C'},\meas_{C'},p)$, the volume convergence result in \cref{thm:volumeconvergence} implies that
\begin{equation}\label{eqn:LEI2New}
\lim_{i \to \infty} \frac{\meas(B_{\rho_i\delta}(\widetilde x_i))}{\omega_n(\rho_i\delta)^n}=\frac{\meas_{C'}(B_\delta(p))}{\omega_n\delta^n}.
\end{equation}
Hence, from \eqref{eqn:LEI1New} and \eqref{eqn:LEI2New} we get that 
\begin{equation}\label{eqn:LEI3New}
    \frac{\meas(B_{\rho_i\delta}(\widetilde x_i))}{\omega_n(\rho_i\delta)^n}\geq \alpha-2\eta, \qquad  \text{for all $i\geq i_0(\eta)$ large enough}.
\end{equation}
Let us fix $R>0$. From the volume convergence result in \cref{thm:volumeconvergence} we have that 
$$
\lim_{i\to \infty} \frac{\meas(B_R(\widetilde x_i))}{\omega_nR^n}= \frac{\meas_\infty(B_R(x_\infty))}{\omega_nR^n}.
$$
By Bishop--Gromov volume comparison on $X$ and \eqref{eqn:LEI3New} applied with $i\geq i_1(\eta,R)$ large enough, we have 
$$
\frac{\meas_\infty(B_R(x_\infty))}{\omega_nR^n}\geq \alpha-3\eta.
$$
Since $R,\eta>0$ are arbitrary we get the sought conclusion taking $R\to +\infty$ and then $\eta\to 0$.
\end{proof}

\begin{proof}[Conclusion of the proof of {\cref{cor:Princ2}}]
It is a direct consequence of \cref{lem:DensitaStaccata},
and \cref{thm:Principale}.
\end{proof}

\subsection{Further existence results}\label{sec:Further}

As a consequence of item (ii) of \cref{lem:DensitaStaccata}, we can derive further existence results on manifolds with nonnegative Ricci curvature and Euclidean volume growth. The next theorem states that if points located at distance $1$ from the tips of every asymptotic cone of a manifold have density $1$, then there exist isoperimetric regions of every volume. This is ultimately due to the fact that item (ii) of \cref{lem:DensitaStaccata} implies that the essentially relevant pGH limits at infinity are $\R^n$, allowing us to conclude by the main existence result in \cite{AFP21}.

\begin{thm}\label{thm:Further1}
Let $k\geq 0$ be a natural number, and let $(M^n,g)=(\mathbb R^k\times N^{n-k},g_{\mathbb R^k}+ g_N)$ be a complete Riemannian manifold such that $\ric\geq 0$ and $\mathrm{AVR}(M^n,g)>0$. Assume 
that for every asymptotic cone $(\mathbb R^k\times C,(0,v_\infty))$ of $M$ and every $p\in \mathbb R^k\times C$ with $\dist_{\mathbb R^k}\otimes\dist_C(p,\mathbb R^k\times\{z_\infty\})=1$, we have
\begin{equation}\label{eqn:DENSITYONE}
\lim_{r\to 0}\frac{\meas_{\mathbb R^k}\otimes\meas_ C(B_r(p))}{\omega_nr^n}=:\vartheta[(\mathbb R^k\times C,\dist_{\mathbb R^k}\otimes\dist_C,\meas_{\mathbb R^k}\otimes\meas_C, p)]=1.
\end{equation}
Then for every volume $V>0$ there exists an isoperimetric region of volume $V$ on $M^n$.
\end{thm}

\begin{proof}
We are in the setting of item (ii) of \cref{lem:DensitaStaccata} with $\alpha=1$. 
From this we deduce that for any pmGH-limit $X_\infty$ at infinity as in the statement of item (ii) of \cref{lem:DensitaStaccata} we have that $\mathrm{AVR}(X_\infty,\dist_\infty,\meas_\infty)\geq 1$. But since $\mathrm{AVR}(X_\infty,\dist_\infty,\meas_\infty)\leq 1$ due to Bishop--Gromov comparison and the fact that $\meas_\infty=\mathcal{H}^n_{d_\infty}$ (cf. \cref{thm:volumeconvergence}, and \cref{rem:PerimeterMMS2}), we have $\mathrm{AVR}(X_\infty,\dist_\infty,\meas_\infty)= 1$. Therefore $(X_\infty,\dist_\infty,\meas_\infty)=(\mathbb R^n,\dist_{\mathbb R^n},\meas_{\mathbb R^n})$ as a consequence of \cite[Corollary 1.7]{DePhilippisGigli18} and the Bishop--Gromov comparison discussed in \cref{rem:PerimeterMMS2}.

Now if $k=0$, this implies that $(M^n, g)$ is GH-asymptotic to flat $\mathbb R^n$ (see \cite[Definition 1.2]{AFP21}), and therefore on $M^n$ we have isoperimetric regions for every volume by \cite[Theorem 5.2]{AFP21}.

More generally, if $k>0$, we can still conclude that on $(M^n, g)$ we have isoperimetric regions of every volume $V>0$ arguing like at the end of the proof of \cref{thm:Principale}. Indeed, pGH limits along sequences $(x_i,y_i)$ with $\{y_i\}_i$ bounded in $N$ are isometric to $(M^n,g)$, and then possible mass of a minimizing sequence lost along these sequence can be obviously recovered as in the proof of \cref{thm:Principale}. On the other hand we obtained that pGH limits obtained along sequences diverging along $N^{n-k}$ are flat $\R^n$, and then a minor modification in the proof of \cite[Theorem 5.2]{AFP21} immediately shows that leak of the mass of a minimizing sequence along such sequences would lead to a contradiction of the minimality assumption on the sequence.
\end{proof}

We record a straightforward consequence of the above statement.

\begin{cor}\label{cor:FurtherExistence}
Let $(M^n,g)$ be a complete Riemannian manifold such that $\ric\geq 0$ and $\mathrm{AVR}(M^n,g)>0$.

Assume that every asymptotic cone of $(M^n, g)$ is \emph{smooth outside the tip}, i.e., it is of the form $C(Z)$ where $Z$ is a smooth closed manifold. Then for every volume $V>0$ there exists an isoperimetric region of volume $V$ on $M^n$.
\end{cor}

\begin{proof}
If every asymptotic cone of $M$ is smooth outside the tip, then \eqref{eqn:DENSITYONE} is clearly satisfied, and thus the existence of isoperimetric regions follows.
\end{proof}

Let us observe that, in the setting of \cref{cor:FurtherExistence}, if $n=2$ then every asymptotic cone is (isometric to) a cone over a circle of some radius, which is obviously smooth outside the tip. Hence the latter statement, when specialized to $n=2$, recovers a particular case of the existence result of \cite{RitoreExistenceSurfaces01} on surfaces with nonnegative curvature.

\subsection{Proof of {\cref{thm1}}}
As anticipated in the Introduction, for the proof of \cref{thm1} we need to show that manifolds with nonnegative sectional curvature have a unique asymptotic cone that splits if and only if the manifold splits.
This is a standard result in the field, but since we were not able to find a satisfactory reference in the literature, we add a short proof for the readers' convenience.
We first prove an auxiliary lemma.

\begin{lemma}\label{lemma:surjectivenonexpansive}
Let $(X,\dist_X)$ and $(Y,\dist_Y)$ be compact metric spaces. If there exist maps $\Phi:X\to Y$ and $\Psi: Y \to X$ surjective and $1$-Lipschitz, then $(X,\dist_X)$ and $(Y,\dist_Y)$ are isometric.
\end{lemma}
\begin{proof}
It is enough to check that $T:=\Psi\circ \Phi$ is an isometry, indeed from
\begin{equation*}
   \dist_X(x,y) = \dist_X(T(x),T(y)) \le \dist_Y(\Phi(x),\Phi(y)) \le \dist_X(x,y) 
\end{equation*}
we deduce that $\Phi$ is an isometry.

Fix distinct $x,y\in X$ and $\eps\in(0,\dist_X(x,y))$. Since $T$ is $1$-Lipschitz, it is enough to prove that $\dist_X(T(x),T(y)) \ge \dist_X(x,y)-\eps$ and then let $\eps\to0$.  Set $D:= \dist_X(x,y)-\eps/2$. 
Given an $\eps/4$-dense set $S \subset X$, i.e., such that for any $z\in X$ there exists $s\in S$ with $\dist_X(z,s)\le \eps/4$, we define $N(S):= \#\{(s_1,s_2)\in S\times S\, : \dist_X(s_1,s_2) \ge D\}$. Let $S_0$ be an $\eps/4$-dense set that minimizes the function $N(\cdot)$ among $\eps/4$-dense sets. Since $T$ is $1$-Lispchitz, then $\dist_X(T(s_1),T(s_2))\le \dist_X(s_1,s_2)$ for any $s_1,s_2\in S_0$, and since $T$ is also surjective, then $T(S_0)$ is $\eps/4$-dense, therefore $N(T(S_0))\ge N(S_0)$. This forces $\dist_X(T(s_1),T(s_2))\ge D$  for any $s_1,s_2\in S_0$ with $\dist_X(s_1,s_2)\ge D$. In particular, if we pick $s_1,s_2\in S_0$ such that $\dist_X(x,s_1)\le \eps/4$ and $\dist_X(y,s_2)\le \eps/4$ we have $\dist_X(s_1,s_2)\ge  -\eps/2 + \dist(x,y) =D $, and therefore, as $T$ is $1$-Lipschitz, we conclude
\begin{equation*}
    \dist_X(T(x),T(y)) \ge \dist_X(T(s_1), T(s_2)) - \eps/2
                       \ge D-\eps/2
                       = \dist(x,y) - \eps\, .
\end{equation*}
\end{proof}

\begin{thm}\label{lem:IfAndOnlyIfSplit}
If $(M^n,g)$ has nonnegative sectional curvature and Euclidean volume growth, then there exists a unique asymptotic cone $(C, \dist,\mathcal{H}^n)$. Moreover $C$ splits if and only if $M$ splits.
\end{thm}

\begin{proof}
Let us begin by proving the uniqueness of asymptotic cones. 

Thanks to \cref{lemma:surjectivenonexpansive}, it is enough to show that, given two asymptotic cones $(C_1, \dist_1, \meas_1)$ and $(C_2, \dist_2, \meas_2)$, there exists a surjective $1$-Lipschitz map $\Psi: \overline B_1(z_1) \to \overline B_1(z_2)$, where $z_1\in C_1$ and $z_2\in C_2$ are tip points.

Fix $p\in M$. Given $1<R_1<R_2$ we consider the set $E\subset B_{R_2}(p)$ of those points $x\in B_{R_2}(p)$ such that there exists a unique unit speed geodesic $\gamma_{p,x}:[0,\dist(x,p)]\to M$ with $\gamma_{p,x}(0)=p$, $\gamma_{p,x}(\dist(x,p))=x$.
Notice that $\vol(B_{R_2}\setminus E)=0$, hence $E\subset B_{R_2}(p)$ is dense.

We then define $T: E\subset B_{R_2}(p) \to B_{R_1}(p)$ as $T(x):=\gamma_{p,x}(\dist(x,p)R_1/R_2)$. Toponogov's theorem implies
\begin{equation}
   \dist(T(x), T(y)) R_1^{-1} \ge  \dist(x,y) R_2^{-1}
   \quad 
   \text{for any $x,y\in E$},
\end{equation}
in particular $T^{-1} : (T(E)\subset B_{R_1}(p), \dist/R_1) \to (B_{R_2}(p), \dist/R_2)$ is $1$-Lipschitz and has dense image, hence it can be extended to a surjective $1$-Lipschitz map $\Psi_{R_1,R_2}: (K_{R_1,R_2}, \dist/R_1) \to (\overline B_{R_2}(p), \dist/R_2)$, where $K_{R_1,R_2}:=\overline{T(E)}$.

Let us now consider two sequences $r_i\to \infty$ and $s_i\to \infty$ realizing $(C_1, \dist_1, \meas_1)$ and $(C_2, \dist_2, \meas_2)$, respectively.
We assume without loss of generality that $ s_i < r_i < s_{i+1}<r_{i+1}$ for any $i\ge 1$. By the Gromov--Hausdorff convergence of the sequences of the rescaled space to the asymptotic cones, by a classical Ascoli--Arzel\`{a}-type argument we can pass to the limit the maps $\Psi_{s_i,r_i}:((K_{s_i,r_i}, \dist /s_i) \to (\overline B_{r_i}(p), \dist/r_i)$ and we derive the existence of a $1$-Lipschitz surjective map
\begin{equation}\label{z6}
    \Psi : K\subset \overline B_1(z_1) \to \overline B_1(z_2) \, ,
\end{equation}
where $z_1\in C_1$ and $z_2\in C_2$ are tip points, and $K$ is compact.
Since $(M^n,g)$ has Euclidean volume growth, the volume convergence theorem, see \cref{thm:volumeconvergence}, guarantees that the reference measures on the asymptotic cones is the Hausdorff measure (with respect to their own distance), and $\mathcal{H}^n((B_1(z_1)) = \meas_1(B_1(z_1)) = \omega_n{\rm AVR}(M^n,g) = \meas_2(B_1(z_2)) = \mathcal{H}^n(B_1(z_2))$, hence
\begin{align*}
\mathcal{H}^n((B_1(z_1))\setminus K) & = \mathcal{H}^n(B_1(z_1))-\mathcal{H}^n(K) 
\\& = \mathcal{H}^n(B_2(z_2)) -\mathcal{H}^n(K)
\\& 
\le \mathcal{H}^n(B_2(z_2)) -\mathcal{H}^n(\Psi(K))
\\& =0 \, ,
\end{align*}
which yields $K=\overline B_1(z_1)$ since $K$ is closed. In particular \eqref{z6} provides the sought map.

\medskip

Let us now prove that the asymptotic cone $(C,\dist, \mathcal{H}^n)$ splits if and only if $M$ splits. In view of the splitting theorem it is enough to show the following implication
\begin{equation}
    C \, \text{contains a line} \implies M \, \text{contains a line,}
\end{equation}
since the other one is readily verified by stability of the product structure under GH-convergence.
Fix $p\in M$. It is enough to show that for any $0<\eps<1/9$ there exist $p_1,p_2\in M$ satisfying $\dist(p_1,p)=\dist(p_2,p) = 1/\eps$ and 
\begin{equation}\label{z222}
    \dist(p, \gamma_{p,p_1}(s)) + \dist(p, \gamma_{p,p_2}(s))
    \le \dist(\gamma_{p,p_1}(s), \gamma_{p,p_2}(s)) + \eps
    \quad \text{for any $s\in (0,1)$}\, ,
\end{equation}
where $\gamma_{p,p_i}:[0,1]\to M$ is a minimizing constant speed geodesic such that $\gamma(0)=p$ and $\gamma(1)=p_i$, for $i=1,2$.

Indeed the curve $\gamma_\eps:(-1/\eps, 1/\eps)\to M$, defined as
\begin{equation}
\gamma_\eps:(-1/\eps, 1/\eps)
     :=
    \begin{cases}
    \gamma_{p,p_1}(-t\eps) & \text{for $t\in (-1/\eps, 0)$}
    \\
    \gamma_{p,p_2}( t\eps) & \text{for $t\in (0, 1/\eps)$}\, ,
    \end{cases}
\end{equation}
is $1$-Lipschitz and satisfies $\gamma_\eps(0)=p$.
Hence, up to extracting a subsequence, $\gamma_\eps \to \gamma$ locally uniformly, where $\gamma:\R \to M$. We claim that $\gamma$ is a line. To see this, we fix $t>0$ and we show that
$\dist(\gamma(t),\gamma(-t)) = 2t$. By \eqref{z222} we have
\[
\begin{split}
    2t &= \dist(p, \gamma_\eps(-t)) + \dist(p,\gamma_\eps(t)) \le  \dist(\gamma_\eps(-t), \gamma_\eps(t)) + \eps \le \dist(\gamma_\eps(-t),p ) + \dist(p, \gamma_\eps(t)) + \eps \\&\le 2t + \eps.
\end{split}
\]
Hence letting $\eps\to0$ implies $2t \le \dist(\gamma(-t), \gamma(t)) \le 2t$ as claimed.

\medskip

So we are left to prove \eqref{z222}.
Given $0<\eps<1/9$ we can find $R=R(\eps,n,p)>5/\eps$ such that
\begin{equation}
    \dist_{GH}(B_R(p), B_R(z))\le \eps^2 R\, ,
\end{equation}
where $z\in C(Z)$ is a tip. Since $C$ splits a Euclidean factor, there exist $q_1,q_2\in B_R(p)$ satisfying
\begin{equation}\label{z1}
    \dist(p,q_1) = R, \quad \dist(p,q_2) = R,\quad \text{and} 
    \quad   \dist(p,q_1) + \dist(p,q_2) \le \dist(q_1,q_2) +  \eps^2 R \, .
\end{equation}
Let $t\in (0,1)$ be such that $\dist(p, \gamma_{p,q_1}(t))=\dist(p,\gamma_{p,q_2}(t))=1/\eps$, where
where $\gamma_{p,q_i}:[0,1]\to M$ is a minimizing constant speed geodesic such that $\gamma(0)=p$ and $\gamma(1)=q_i$, for $i=1,2$. We set $p_i:=\gamma_{p,q_i}(t)$ and check that
\begin{equation}
    \dist(p, \gamma_{p,q_1}(s)) + \dist(p, \gamma_{p,q_2}(s))
    \le \dist(\gamma_{p,q_1}(s), \gamma_{p,q_2}(s)) + \eps
    \quad \text{for any $s\in (0,t)$}\, ,
\end{equation}
which amounts to our conclusion.

Toponogov's theorem implies that $\dist(\gamma_{p,q_1}(s), \gamma_{p,q_2}(s))\ge |\tilde \gamma_{\tilde p,\tilde q_1}(s) - \tilde \gamma_{\tilde p,\tilde q_2}(s)|$, where $(\tilde p, \tilde q_1, \tilde q_2)$ is a comparison triangle in $\R^2$ corresponding to the triangle $(p,q_1,q_2)$, and $\tilde\gamma_{\tilde p,\tilde q_i}$ is the constant speed Euclidean segment from $\tilde p$ to $\tilde q_i$, for $i=1,2$. The similarity of the triangles $(\tilde p, \tilde q_1, \tilde q_2)$ and $(\tilde p, \tilde \gamma_{\tilde p,\tilde q_1}(s), \tilde \gamma_{\tilde p,\tilde q_2}(s))$, together with \eqref{z1}, give
\begin{align*}
    \dist(\gamma_{p,q_1}(s), \gamma_{p,q_2}(s)) & \ge 
    |\tilde \gamma_{\tilde p,\tilde q_1}(s) - \tilde \gamma_{\tilde p,\tilde q_2}(s)|
      =  |\tilde q_1 - \tilde q_2|\frac{ |\tilde p - \tilde \gamma_{\tilde p,\tilde q_1}(s)|}{|\tilde p - \tilde q_1|}
    \\& = \dist(q_1,q_2)\frac{\dist(p,\gamma_{p,q_1}(s)) + \dist(p,\gamma_{p,q_2}(s))}{2R}
    \\& \ge (\dist(p,q_1) + \dist(p,q_2) -\eps^2 R)\frac{\dist(p,\gamma_{p,q_1}(s)) + \dist(p,\gamma_{p,q_2}(s))}{2R}
    \\& = (1-\frac{\eps^2}{2})( \dist(p,\gamma_{p,q_1}(s)) + \dist(p,\gamma_{p,q_2}(s)))
    \\& \ge  \dist(p,\gamma_{p,q_1}(s)) + \dist(p,\gamma_{p,q_2}(s))
    -\eps\, ,
\end{align*}
where in the last step we used $\dist(p,\gamma_{p,q_1}(s))= \dist(p,\gamma_{p,q_2}(s))\le 1/\eps$.
\end{proof}

\begin{remark}
The statement of \cref{lem:IfAndOnlyIfSplit} holds in the class of Alexandrov spaces with nonnegative sectional curvature. Indeed, the proof only builds on Toponogov's theorem and metric arguments. 
\end{remark}

\begin{proof}[Conclusion of the proof of {\cref{thm1}}]
The proof immediately follows from \cref{lem:IfAndOnlyIfSplit} and \cref{cor:Princ2}.
\end{proof}




\printbibliography[title={References}]

\typeout{get arXiv to do 4 passes: Label(s) may have changed. Rerun} 

\end{document}